\documentclass[11pt]{article}

\newif\ifDRAFT
\DRAFTfalse

\usepackage{times}
\usepackage{graphicx} 
\usepackage{epsfig} 
\usepackage{subfigure}

\usepackage{algorithm}
\usepackage{algorithmic}

\usepackage[T1]{fontenc}
\usepackage[latin9]{inputenc}
\usepackage{geometry}
\usepackage{color}
\usepackage{float}
\usepackage{url}
\usepackage{bm}
\usepackage{amsthm}
\usepackage{amsmath}
\usepackage{amssymb}
\usepackage{graphicx}
\usepackage{esint}
\usepackage{paralist}
\usepackage{placeins}
\usepackage{epstopdf}
\usepackage{array}
\usepackage{booktabs}
\usepackage{rotating}

\ifDRAFT
\newcommand{\marrow}{\marginpar[\hfill$\longrightarrow$]{$\longleftarrow$}}
\newcommand{\niceremark}[3]
   {\textcolor{red}{\textsc{#1 #2:} \marrow\textsf{#3}}}
\newcommand{\niceremarkblue}[3]
   {\textcolor{blue}{\textsc{#1 #2:} \marrow\textsf{#3}}}
\newcommand{\Ken}[2][says]{\niceremark{Ken}{#1}{#2}}
\newcommand{\David}[2][says]{\niceremark{David}{#1}{#2}}
\newcommand{\Haim}[2][says]{\niceremarkblue{Haim}{#1}{#2}}
\usepackage[inline]{showlabels}
\else
\newcommand{\Ken}[1]{}
\newcommand{\David}[1]{}
\newcommand{\Haim}[1]{}
\fi

\newcommand{\noun}[1]{\textsc{#1}}
\providecommand{\tabularnewline}{\\}
\floatstyle{ruled}
\newfloat{algorithm}{tbp}{loa}
\providecommand{\algorithmname}{Algorithm}
\floatname{algorithm}{\protect\algorithmname}

  \theoremstyle{remark}
  \newtheorem*{rem*}{\protect\remarkname}
\theoremstyle{plain}
\newtheorem{thm}{\protect\theoremname}
  \theoremstyle{plain}
  \newtheorem{prop}[thm]{\protect\propositionname}
  \theoremstyle{plain}
  \newtheorem{lem}[thm]{\protect\lemmaname}
\theoremstyle{definition}
\newtheorem{defn}{\protect\definitionname}

\providecommand{\lemmaname}{Lemma}
\providecommand{\propositionname}{Proposition}
\providecommand{\remarkname}{Remark}
\providecommand{\theoremname}{Theorem}
\providecommand{\definitionname}{Definition}

\title{Faster Kernel Ridge Regression Using Sketching and Preconditioning}

\author{
Haim Avron \\
Tel Aviv University\\
haimav@post.tau.ac.il \\
\and
Kenneth L. Clarkson \\
IBM Almaden Research Center \\
klclarks@us.ibm.com \\
\and
David P. Woodruff \\
IBM Almaden Research Center \\
dpwoodru@us.ibm.com
}

\begin{document}

\maketitle

\global\long\def\R{\mathbb{R}}

\global\long\def\H{{\cal H}}

\global\long\def\X{{\cal X}}

\global\long\def\Y{{\cal Y}}

\global\long\def\e{{\mathbf{e}}}

\global\long\def\et#1{{\e(#1)}}

\global\long\def\ef{{\mathbf{\et{\cdot}}}}

\global\long\def\x{{\mathbf{x}}}

\global\long\def\q{{\mathbf{q}}}

\global\long\def\xt#1{{\x(#1)}}

\global\long\def\xf{{\mathbf{\xt{\cdot}}}}

\global\long\def\d{{\mathbf{d}}}

\global\long\def\b{{\mathbf{b}}}

\global\long\def\u{{\mathbf{u}}}

\global\long\def\y{{\mathbf{y}}}

\global\long\def\w{{\mathbf{w}}}

\global\long\def\yt#1{{\y(#1)}}

\global\long\def\yf{{\mathbf{\yt{\cdot}}}}

\global\long\def\z{{\mathbf{z}}}

\global\long\def\v{{\mathbf{v}}}

\global\long\def\h{{\mathbf{h}}}

\global\long\def\s{{\mathbf{s}}}

\global\long\def\c{{\mathbf{c}}}

\global\long\def\p{{\mathbf{p}}}

\global\long\def\f{{\mathbf{f}}}

\global\long\def\rb{{\mathbf{r}}}

\global\long\def\rt#1{{\rb(#1)}}

\global\long\def\rf{{\mathbf{\rt{\cdot}}}}

\global\long\def\mat#1{{\ensuremath{\bm{\mathrm{#1}}}}}

\global\long\def\matN{\ensuremath{{\bm{\mathrm{N}}}}}

\global\long\def\matX{\ensuremath{{\bm{\mathrm{X}}}}}

\global\long\def\matA{\ensuremath{{\bm{\mathrm{A}}}}}

\global\long\def\matB{\ensuremath{{\bm{\mathrm{B}}}}}

\global\long\def\matC{\ensuremath{{\bm{\mathrm{C}}}}}

\global\long\def\matD{\ensuremath{{\bm{\mathrm{D}}}}}

\global\long\def\matP{\ensuremath{{\bm{\mathrm{P}}}}}

\global\long\def\matU{\ensuremath{{\bm{\mathrm{U}}}}}

\global\long\def\matM{\ensuremath{{\bm{\mathrm{M}}}}}

\global\long\def\matR{\mat R}

\global\long\def\matS{\mat S}

\global\long\def\matY{\mat Y}

\global\long\def\matI{\mat I}

\global\long\def\matJ{\mat J}

\global\long\def\matZ{\mat Z}

\global\long\def\matV{\mat V}

\global\long\def\matL{\mat L}

\global\long\def\matQ{\mat Q}

\global\long\def\matK{\mat K}

\global\long\def\matH{\mat H}

\global\long\def\S#1{{\mathbb{S}_{N}[#1]}}

\global\long\def\IS#1{{\mathbb{S}_{N}^{-1}[#1]}}

\global\long\def\PN{\mathbb{P}_{N}}

\global\long\def\TNormS#1{\|#1\|_{2}^{2}}

\global\long\def\TNorm#1{\|#1\|_{2}}

\global\long\def\InfNorm#1{\|#1\|_{\infty}}

\global\long\def\FNorm#1{\|#1\|_{F}}

\global\long\def\FNormS#1{\|#1\|^{2}_{F}}

\global\long\def\UNorm#1{\|#1\|_{\matU}}

\global\long\def\UNormS#1{\|#1\|_{\matU}^{2}}

\global\long\def\UINormS#1{\|#1\|_{\matU^{-1}}^{2}}

\global\long\def\ANorm#1{\|#1\|_{\matA}}

\global\long\def\XNorm#1#2{\|#1\|_{#2}}

\global\long\def\BNorm#1{\|#1\|_{\mat B}}

\global\long\def\ANormS#1{\|#1\|_{\matA}^{2}}

\global\long\def\AINormS#1{\|#1\|_{\matA^{-1}}^{2}}

\global\long\def\T{\textsc{T}}

\global\long\def\pinv{\textsc{+}}

\global\long\def\Expect#1{{\mathbb{E}}\left[#1\right]}

\global\long\def\ExpectC#1#2{{\mathbb{E}}_{#1}\left[#2\right]}

\global\long\def\dotprod#1#2#3{(#1,#2)_{#3}}

\global\long\def\dotprodsqr#1#2#3{(#1,#2)_{#3}^{2}}

\global\long\def\Trace#1{{\bf Tr}\left(#1\right)}

\global\long\def\range#1{{\bf Range}\left(#1\right)}

\global\long\def\span#1{{\bf Sp}\left(#1\right)}

\global\long\def\nnz#1{{\bf nnz}\left(#1\right)}

\global\long\def\vol#1{{\bf vol}\left(#1\right)}

\global\long\def\sign#1{{\bf sign}\left(#1\right)}

\global\long\def\poly#1{{\bf poly}\left(#1\right)}

\begin{abstract}
  Kernel Ridge Regression is a simple yet powerful technique for non-parametric
  regression whose computation amounts to solving a linear system. This system is usually dense and
  highly ill-conditioned. In addition, the dimensions of the matrix are the same as the number of
  data points, so direct methods are unrealistic for large-scale datasets.
  In this paper, we propose a preconditioning technique for accelerating the solution of the
  aforementioned linear system. The preconditioner is based on random feature maps, such as random
  Fourier features, which have recently emerged as a powerful
  technique for speeding up and scaling the training of kernel-based methods, such as
  kernel ridge regression, by resorting to approximations. However, random feature maps only provide
  crude approximations to the kernel function, so delivering state-of-the-art results by directly solving the approximated system requires the number of random
  features to be very large. We show that random feature maps can be much more effective in forming
  preconditioners, since under certain conditions a not-too-large number of random features
  is sufficient to yield an effective preconditioner. We empirically evaluate our method
  and show it is highly effective for datasets of up to one million training examples.

\end{abstract}

\section{Introduction}
\label{sec:introduction}

{\em Kernel Ridge Regression (KRR)} is a simple yet powerful technique for non-parametric
regression whose computation amounts to solving a linear system. Its underlying mathematical framework
is as follows. A kernel function, $k : {\cal X} \times {\cal X} \to \R$, is defined on the
input domain ${\cal X} \subseteq \R^d$. The kernel function $k$ may be (non-uniquely) associated with an
embedding of the input space into a high-dimensional Reproducing Kernel Hilbert Space ${\cal H}_k$ (with
inner product $\langle \cdot, \cdot \rangle_{\H_k}$) via a feature map, $\Psi:
\X \to \H_k$ such that $k(\x, \z) = \langle \Psi(\x), \Psi(\z)\rangle_{\H_k}$. Given training
data $(\x_1, y_1),\dots,(\x_n,y_n)\in \X\times\Y$ and ridge parameter $\lambda$,
we perform linear ridge regression on $(\Psi(\x_1),y_1),\dots,(\Psi(\x_n),y_n)$.
Ultimately, the model has the form
\begin{equation*}
f(\x) = \sum^n_{i=1} c_i k(\x_i, \x)
\end{equation*}
where $c_1, \dots,c_n$ can be found by solving the linear equation
\begin{equation}
  \label{eq:krr}
  (\matK + \lambda \matI_n)\c = \y\,.
\end{equation}
Here $\matK\in\R^{n\times n}$ is the {\em kernel matrix} or {\em Gram matrix} defined by $\matK_{ij}\equiv k(\x_i, \x_j)$,
$\c = [c_1 \cdots c_n]^\T$ and $\y \equiv [y_1 \cdots y_n]^\T$.
See Saunders et al.~\cite{SGV98} for details.

Compared to Kernel Support Vector Machines (KSVM), the computations involved in KRR are conceptually much simpler:
solving a single linear system as opposed to solving a convex quadratic optimization
problem. However, KRR has been observed experimentally to often perform just as well as KSVM~\cite{FM01}.
In this paper, we exploit the conceptual simplicity of KRR, and using advanced techniques in
numerical linear algebra design an efficient method for solving~\eqref{eq:krr}.

For widely used kernel functions, the kernel matrix $\matK$ is fully dense, so solving~\eqref{eq:krr} using standard direct methods
takes $\Theta(n^3)$, which is prohibitive even for modest $n$. Although statistical analysis does suggest
that iterative methods can be stopped early~\cite{BK10}, the condition number of $\matK$ tends
to be so large that a large number of iterations are still necessary. Thus, it is not surprising that the literature
has moved towards designing approximate methods.

One popular strategy is the Nystr\"{o}m method~\cite{WilliamsSeeger00} and variants that
improve the sampling process~\cite{KMT12, GM13}.
We also note recent work by Yang et al.~\cite{YPM15} that replace
the sampling with sketching. More relevant to this paper is the line of
research on randomized construction of approximate feature maps, originating from the seminal
work of Rahimi and Recht~\cite{RahimiRecht07}.
The underlying idea is to
construct a distribution $D$ on functions from $\X$ to $\R^s$ ($s$ is a parameter) such that
\begin{equation*}
k(\x,\z) = \ExpectC{\varphi \sim D}{\varphi(\x)^\T \varphi(\z)}\,.
\end{equation*}
One then samples a $\varphi$ from $D$ and uses $\tilde{k}(\x, \z)\equiv \varphi(\x)^\T \varphi(\y)$
as an alternative kernel. Training can be done in $O(ns^2 + T_\varphi(\x_1,\dots,\x_n))$ where $T_\varphi(\x_1,\dots,\x_n)$ is the time
required to compute $\varphi(\x_1),\dots,\varphi(\x_n)$. This technique  has been used in recent years
to obtain state-of-the-art accuracies for some important datasets~\cite{HuangEtAl14, DaiEtAl14, AvronSindhwani14,ChenEtAl16}.
We also note recent work on using random features to construct stochastic gradients~\cite{LuEtAl14}.

One striking feature of the aforementioned papers is the use of a very large number of random features.
Random feature maps typically provide only crude approximations to the kernel function, so
to approach the full capacity of exact kernel learning (which is required to obtain state-of-the-art
results) many random features are necessary. Nevertheless, even with a very large number of
random features, we sometimes pay a price in terms of generalization performance. Indeed, in section~\ref{sec:vs-rfm}
we show that in some cases, driving $s$ to be as large as $n$ is not sufficient to achieve
the same test error rate as that of the full kernel method. Ultimately, methods that use approximations compromise
in terms of performance in order to make the computation tractable.

\subsection{Contributions}

We propose to use random feature maps as a means of forming a preconditioner for
the kernel matrix. This preconditioner can be used to solve~\eqref{eq:krr} to high accuracy using an iterative method.
Thus, while training time still benefits from the use of high-quality random feature
maps, there is no compromise in terms of modeling capabilities, modulo the decision
to use kernel ridge regression and not some other learning method.

We provide a theoretical analysis that shows that at least for one kernel selection, the polynomial kernel,
selecting the number of random features to be proportional to the so-called \emph{statistical dimension}
of the problem (this classical quantity is also frequently referred to as the \emph{effective degrees-of-freedom})
yields a high-quality preconditioner in the sense that the relevant condition number is bounded by a constant.
These can be viewed as a generalization of recent results on sharper bounds for linear regression
and low-rank approximation with regularization~\cite{ACW16}.
In addition, we discuss a method for testing whether the preconditioner computed by our algorithm is indeed such
that the relevant condition number is bounded. While our analytical results are mostly of theoretical value (e.g. they
are limited to only the polynomial kernel and are likely very pessimistic), they do expose an important connection
between the statistical dimension and preconditioner size.

Finally, we report experimental results with a distributed-memory parallel implementation of our algorithm. We
demonstrate the effectiveness of our algorithm for training high-quality models using both
the Gaussian and polynomial kernel on datasets as large as one million training examples without
compromising in terms of statistical capacity. For example, on one dataset with one million examples
our code is able to solve~\eqref{eq:krr} to relatively high accuracy in about an hour on resources readily
available to researchers and practitioners (a cluster of EC2 instances).

An open-source implementation of the algorithm is available through the libSkylark
library (\url{http://xdata-skylark.github.io/libskylark/}).


\subsection{Related Work}

Devising scalable methods for kernel methods has long been an active research topic.
In terms of approximations, one dominant line of work constructs low-rank approximations of the Gram matrix.
Two popular variants of this approach, as mentioned earlier, are the randomized feature maps and
the Nystr\"{o}m methods~\cite{WilliamsSeeger00}. There are many variants of these approximation schemes, and it is outside
the scope of this paper to mention all of them. Recent work has also focused on devising scalable methods
that are capable of utilizing rather high rank approximations~\cite{DaiEtAl14, AvronSindhwani14, LuEtAl14}.
In contrast, our goal is to develop a method which is capable of using lower rank approximation without paying
a price in terms of model quality.

Another approach is to approximate the kernel matrix so that it will be more amenable to matrix-vector products or
linear system solution. One idea is to use the Fast Gauss Transform to accelerate the matrix-vector products of the kernel matrix
by an arbitrary vector~\cite{RD06, MorariuEtAl08}.
Another approach is to use a tree code to efficiently perform matrix-vector products~\cite{MorariuEtAl08, CWA14}. Related is also the
hierarchical matrix approach in which an hierarchical matrix approximation to the kernel matrix is built~\cite{CAS16}. This
representation is amenable to efficient implementation of wide range of operations on the approximate kernel matrix, including matrix-vector product
and linear system solution.

The preconditioning approach has also been explored in the literature. Srinivasan et al. propose to use a regularized
kernel matrix as a preconditioner in a flexible Krylov method~\cite{SrinivasanEtAl14}.
The regularized kernel matrix, which has a lower condition number,  is solved using an inner conjugate gradient iteration.
In parallel work to ours, Cutajar et al. recently discussed various preconditioning techniques for kernel matrices~\cite{COCF16}.
One of the methods they propose is using random features to form a preconditioner. However, unlike our work, they
do not include any theoretical analysis of this preconditioning approach. Furthermore, we propose additional
algorithmic enhancements (multiple level preconditioning, testing preconditioners). Finally, it is worth mentioning
that Cutajar et al. only experiment with small-scale low-dimensional datasets, while we present experimental results
with large-scale high-dimensional datasets.

For a broad discussion of scalable methods for kernel learning, including of ideas not mentioned here, see Bottou et al.~\cite{Bottou2007}.
\section{Preliminaries}
\label{sec:preliminaries}

\subsection{Basic Definitions and Notation}

We denote scalars using Greek letters or using $x,y,\dots$. Vectors
are denoted by $\x,\y,\dots$ and matrices by $\matA,\mat B,\dots$.
The $s\times s$ identity matrix is denoted $ \matI_s$.
We use the convention that vectors are column-vectors. We use $\nnz{\cdot}$ to
denote the number of nonzeros in a vector or matrix. We denote by $[n]$ the
set ${1,\dots,n}$. The notation $\alpha = (1 \pm \gamma)\beta$ means
that $(1- \gamma)\beta \leq \alpha \leq (1 + \gamma)\beta$.

A symmetric matrix $\matA$ is positive semi-definite (PSD) if
$\x^\T \matA \x \geq 0$ for every vector $\x$. It is positive definite (PD) if
$\x^\T \matA \x > 0$ for every vector $\x \neq 0$.
For any two symmetric matrices $\matA$ and $\matB$ of the same size, $\matA \preceq \matB$ means
that $\matB - \matA$ is a PSD matrix.
For a PD matrix $\matA$, $\ANorm{\cdot}$ denotes the
norm induced by $\matA$, i.e. $\ANormS{\x}\equiv\x^\T \matA \x$.

We denote the training set by $(\x_1, y_1), \dots, (\x_n, y_n) \in \X \times \Y \subseteq \R^d \times \R$.
Note that $n$ denotes the number of training examples, and $d$ their dimension.
We denote the kernel, which is a function from $\X \times \X$ to $\R$, by $k$.
We denote the kernel matrix by $\matK$, i.e. $\matK_{ij} = k(\x_i, \x_j)$.
The associated Reproducing Kernel Hilbert Space (RKHS) is denoted by ${\cal H}_k$,
and the associated inner product by $\dotprod{\cdot}{\cdot}{{\cal H}_k}$.
We use $\lambda$ to denote the ridge regularization parameter, which we always
assume to be greater than 0.

\subsection{Random Feature Maps}
\label{sec:rfm}

As explained in the introduction, a {\em random feature map} is a
distribution $D$ on functions from $\X$ to $\R^s$ such that
\begin{equation*}
k(\x,\z) = \ExpectC{\varphi \sim D}{\varphi(\x)^\T \varphi(\z)}\,.
\end{equation*}
Throughout the paper, we use $s$ to denote the number of random features. This
quantity is a parameter of the various algorithms.

In recent years, diverse uses of random features have been seen for a wide spectrum of techniques and problems.
However, the original motivation is the following technique, originally due to Rahimi and Recht~\cite{RahimiRecht07},
which we refer to as the {\em Random Features Method}:
sample a $\varphi$ from $D$ and use $\tilde{k}(\x, \z)\equiv \varphi(\x)^\T \varphi(\z)$
as an alternative kernel. For KRR the resulting model is
\begin{equation*}
f(\x) = \varphi(\x)^\T (\matZ^\T \matZ + \lambda \matI_s)^{-1}\matZ^\T \y
\end{equation*}
where $\matZ \in \R^{n\times s}$ has $i$'th row $\z_i = \varphi(\x_i)$.
Thus, training can be done in $O(ns^2 + T_\varphi(\x_1,\dots,\x_n))$ where $T_\varphi(\x_1,\dots,\x_n)$ is the time
required to compute $\z_1,\dots,\z_n$.

Although our proposed method can be composed with any of the many feature maps
suggested in recent literature, we discuss and experiment with two specific feature maps.
The first is {\em random Fourier features} originally suggested by Rahimi and Recht~\cite{RahimiRecht07}.
The underlying observation that lead to this transform is that a shift-invariant kernel\footnote{
That is, it is possible to write $k(\x,\z)=k_0(\x-\z)$ for some positive definite function $k_0:\R^d\rightarrow\R$.
} $k$ for which $k(\x,\z)=1$ for all $\x\in\X$ can be expressed as
\begin{equation*}
k(\x, \z) = \ExpectC{\w \sim p, b \sim U(0,2\pi)}{\cos(\w^\T\x + b) \cos(\w^\T\y + b)}
\end{equation*}
where $p$ is some appropriate distribution that depends on the kernel function (e.g. Gaussian distribution
for the Gaussian kernel). The existence of such a $p$ for every shift-invariant kernel function $k$ is a
consequence of Bochner's Theorem; see Rahimi and Recht~\cite{RahimiRecht07}
for details. The feature map is then a Monte-Carlo sample:
$\varphi(\x) = s^{-1/2} [\cos(\w^\T_1\x + b_1) \dots \cos(\w^\T_s\x + b_n)]^\T$
where $\w_1,\dots,\w_s$ are sampled from $p$ and $b_1,\dots,b_n$ are sampled from a uniform
distribution on $[0,2\pi]$.

The second feature map we discuss in this paper is {\sc TensorSketch}~\cite{Pagh13}, which
is designed to be used for the polynomial kernel $k(\x,\z)=(\x^\T\z)^q$.
In order to describe {\sc TensorSketch} we first describe {\sc CountSketch}~\cite{CCF04}.
Suppose we want to sketch a $d$ dimensional vector to an $s$ dimensional vector.
{\sc CountSketch} is specified by a $2$-wise independent hash function $h:[d] \rightarrow [s]$
and a $2$-wise independent sign function $g:[d] \rightarrow \{+1,-1\}$.
Suppose {\sc CountSketch} is applied to a vector $\x\in\R^d$ to yield $\z\in\R^s$.
The value of coordinate $i$ of $\z$ is $\sum_{j \mid h(j) = i} g(j) x_j$
It is clear that this transformation can be represented
as a $s \times d$ matrix in which the $j$-th column contains a single non-zero entry $g(j)$
in the $h(j)$-th row. Therefore, the distribution on $h$ and $g$ defines a distribution on
$s \times d$ matrices.

{\sc TensorSketch} implicitly defines a random linear transformation
$\matS \in \R^{s\times d^q}$. The transform is
specified using $q$ $3$-wise independent hash
functions $h_1, \ldots, h_q: [d] \rightarrow [s]$,
and $q$ $4$-wise independent sign functions $g_1, \ldots, g_q: [d] \rightarrow \{+1, -1\}$.
The {\sc TensorSketch} matrix $\matS$ is then a  {\sc CountSketch} matrix with hash function
$H:[d]^q \rightarrow [s]$ and sign function $G:[d]^q \rightarrow \{+1, -1\}$ defined as follows:
\begin{equation*}
H(i_1, \ldots, i_q) \equiv h_1(i_1) + h_2(i_2) + \cdots + h_q(i_q) \bmod m,
\end{equation*}
and
\begin{equation*}
G(i_1, \ldots, i_q) \equiv g_1(i_1) \cdot g_2(i_1) \cdots g_q(i_q)\,.
\end{equation*}
We now index the columns of $\matS$ by $[d]^q$ and set column $(i_1, \ldots, i_q)$
to be equal to $G(i_1, \ldots, i_q)\cdot \e_{H(i_1, \ldots, i_q)}$ where $\e_j$ denotes
the $j$th identity vector.
Let $v_q : \R^d \to \R^{d^q}$ map each vector to the evaluation of all possible degree $q$ monomials
of the entries. Thus, $k(\x,\z) = v_q(\x)^\T v_q(\z)\approx v_q(\x)^\T\matS^\T \matS v_q(\z)$.
The feature map is then defined by $\varphi(\x) = \matS v_q(\x)$.

A crucial observation that makes this transformation useful is that via a clever application of the Fast Fourier Transform,
$\varphi(\x)$ can be computed in $O(q(\nnz{\x} + s \log{s}))$ (see Pagh~\cite{Pagh13} for details),
which allows for a fast application of the transform.

\subsection{Fast Numerical Linear Algebra Using Sketching}

Sketching has recently emerged as a powerful dimensionality reduction
technique for accelerating numerical linear algebra primitives typically
encountered in statistical learning such as linear regression,
low rank approximation, and principal component analysis.
The following description is only a brief semi-formal description of this emerging area.
We refer the interested reader to a recent surveys~\cite{Woodruff14, YMM15} for more information.

The underlying idea is to construct an embedding of a high-dimensional
space into a lower-dimensional one, and use this to accelerate the computation.
For example, consider the classical linear regression problem
\begin{equation*}
\label{eq:regression_prob}
\w^{\star}=\arg\min_{\w\in\R^{d}}\TNorm{\matX\w-\y},
\end{equation*}
where $\matX\in\R^{n\times d}$ is a sample-by-feature design matrix,
and $\y\in\R^{n}$ is the target vector. Sketching methods for linear regression
define a distribution on $s$-by-$n$ matrices, sample
a matrix $\matS$ from this distribution and solve the approximate problem
\begin{equation*}
\w=\arg\min_{\w\in\R^{d}}\TNorm{\matS\matX\w-\matS\y}\,.
\end{equation*}
If the distribution is an {\em Oblivious Subspace Embedding (OSE)} for $\range{[\matX\,\y]}$,
i.e. with high probability for all $\x \in \range{[\matX\,\y]}$ we have $\TNorm{\matS \x} = (1\pm \epsilon)\TNorm{\x}$,
then one can show that with high probability $\TNorm{\matX\w-\y}\leq(1+\epsilon)\TNorm{\matX\w^{\star}-\y}$.
See Drineas et al.~\cite{DrineasEtAl11}.

The ``sketch-and-solve'' approach just described allows only crude approximations: the $\epsilon$-dependence
for OSEs is $\epsilon^{-2}$. There is also an alternative ``sketch-to-precondition'' approach, which
enjoy a much better $\log(1/\epsilon)$ dependence for $\epsilon$ and so
supports very high accuracy approximations.

For linear regression the idea is as follows. The sketched
matrix is factored, $\matS \matX  =\matQ\matR$, and the un-sketched original problem
is solved using an iterative method, with $\matR$ serving as a preconditioner.
This technique has been shown to be very effective in solving linear regression to high
accuracy~\cite{AMT10,MSM14}.

\subsection{Random Feature Maps as Sketching}

A random feature map $\varphi$ can naturally be extended from defining a function
$\X \to \R^s$ to one that defines a function
$\span{\{k(\x_1, \cdot), \dots, k(\x_n, \cdot) \} } \to \R^s$ via
\begin{equation*}
\varphi\left( \sum^n_{i=1}\alpha_ik(x_i, \cdot)\right) \equiv \sum^n_{i=1}\alpha_i \varphi(\x_i)\,.
\end{equation*}
Thus, random feature maps can be viewed as a sketch that embeds $$\span{\{k(\x_1, \cdot), \dots, k(\x_n, \cdot) \} }$$
 in $\R^s$ (here $\span{\cdot}$ denotes the span of set of functions $\X \to \R$). At least in one case, this embedding is an OSE, allowing stronger analysis of algorithms involving
such feature maps: {\sc TensorSketch} defines an OSE~\cite{ANW14}.

Viewed this way, the random features method is a ``sketch-and-solve'' approach. It is therefore
not surprising that it produces suboptimal models. In this paper, we take the ``sketch-to-precondition''
approach to utilizing sketching.

\section{Random Features Preconditioning}
\label{sec:algorithm}

\subsection{Algorithm}

We propose to use the random feature maps to form a preconditioner for
$\matK + \lambda \matI_n$.
In particular, let $\matZ \in \R^{n\times s}$ have rows $\z^\T_1,\dots,\z^\T_n \in \R^s$
  where $\z_i = \varphi(\x_i)$. Here and throughout the rest of the paper,  $\varphi$ is a sample from
  the distribution defined by the random feature map.
  We assume that  $s < n$. We solve $(\matK + \lambda \matI_n)\c = \y$ using
  Preconditioned Conjugate Gradients (PCG) with $\matZ \matZ^\T + \lambda \matI_n$
  as a preconditioner.

  Using $\matZ \matZ^\T + \lambda \matI_n$ as preconditioner requires in each iteration
  the computation of $(\matZ \matZ^\T + \lambda \matI_n)^{-1}\x$ for some vector $\x$.
  To do this, in preprocessing we use the Woodbury formula and the Cholesky decomposition
  $\matL \matL^\T = \matZ^\T \matZ + \lambda \matI_s$ to obtain
  \begin{equation*}
  (\matZ \matZ^\T + \lambda \matI_n)^{-1}
  	= \lambda^{-1}\left(\matI_n - \matZ \left(\matZ^\T \matZ + \lambda \matI_s\right)^{-1}\matZ^\T\right)
	= \lambda^{-1}\left(\matI_n -\matU^\T \matU \right),
  \end{equation*}
  where $\matU = \matL^{-\T}\matZ^\T$. 
  While we could use the Cholesky decomposition $\matL$ to solve for
$\matZ\matZ^\T + \lambda \matI_n$ directly, we empirically observed that forming $\matU$
and using the above formula reduces the cost per iteration considerably
and more than compensates for the additional time spent on pre-processing.
Overall, the end result is a considerable improvement in running time in our implementation
(even though both approaches have the same asymptotic complexity).
One possible reason is that with our scheme we only need to do matrix-matrix products
(GEMM operations) to apply the preconditioner, and we avoid triangular solves (TRSM operation)
which tends not to exhibit good parallel scalability (we observed gains mostly on large
datasets where a large number of processes were used), although another possible reason might be specific
features of the underlying library used for parallel matrix operations (Elemental~\cite{PoulsonEtAl13}) and with another
library it might be preferable to use $\matL$ to solve for $\matZ\matZ^\T + \lambda \matI_n$ directly.
%
%
%
  A pseudocode description of the algorithm appears as Algorithm~\ref{alg:fastkrr}.

  Before analyzing the quality of the preconditioner we discuss the complexities of various operations associated with the algorithm.
  The cost of computing the kernel matrix $\matK$ depends on the kernel and the sparsity
  of the input data $\x_1,\dots,\x_n$. For the Gaussian kernel and the polynomial kernel
  the matrix can be computed in $O(n\sum^n_{i=1}\nnz{\x_i})$ time (although in many cases
  it might be beneficial to use the straightforward $\Theta(n^2 d)$ algorithm).
  The cost of computing $\z_1,\dots,\z_n$ depends on the specific kernel and feature
  map used as well. For the Gaussian kernel with random Fourier features it is $O(s\sum^n_{i=1}\nnz{\x_i})$,
  and for the polynomial kernel with \noun{TensorSketch} it is $O(q(\sum^n_{i=1}\nnz{\x_i} + ns\log{s}))$.
  Computing and decomposing $\matZ^\T \matZ + \lambda \matI_s$ takes $\Theta(ns^2)$, and then another $\Theta(ns^2)$
  for computing $\matU$. The dominant cost per iteration is now multiplying the kernel matrix by a vector, which
  is $\Theta(n^2)$ operations.

\begin{algorithm}[t]
\begin{algorithmic}[1]

  \STATE \textbf{Input: }Data $(\x_1,y_1),\dots, (\x_n,y_n)\in \R^d \times \R$, kernel $k(\cdot, \cdot)$, feature map
  generating algorithm $T$, $\lambda > 0$, $s < n$, accuracy parameter $\epsilon > 0$.

  \STATE

  \STATE Compute kernel matrix $\matK \in \R^{n \times n}$.

  \STATE Using $T$, create a feature map $\varphi : \R^d \to \R^s$.

  \STATE Compute $\z_i = \varphi(\x_i), i=1,\dots,n$ and stack them in a
  matrix $\matZ$.

  \STATE Compute $\matZ^\T \matZ$, and Cholesky decomposition
  $\matL^\T \matL = \matZ^\T \matZ + \lambda \matI_s$.

  \STATE Solve $\matL^{\T} \matU = \matZ^\T$ for $\matU$.

  \STATE Stack $y_1,\dots,y_n$ in a vector $\y\in\R^n$.

  \STATE Solve $(\matK + \lambda \matI_n)\c = \y$ using PCG to accuracy $\epsilon$ with
  $\matZ \matZ^\T + \lambda \matI_n$ as a preconditioner. In each iteration, in order to apply
  the preconditioner to some $\x$, compute $\lambda^{-1}(\x - \matU^\T \matU \x)$.

  \STATE

  \STATE {\bf return} $\tilde{\c}$

\end{algorithmic}
\protect\caption{\label{alg:fastkrr}Faster Kernel Ridge Regression using Sketching and Preconditioning.}
\end{algorithm}

\subsection{Analysis}
\label{sec:analysis}

We now analyze the algorithm
when applied to kernel ridge regression with the polynomial kernel and  \noun{TensorSketch}
as the feature map.  In particular, in the following theorem we show that if $s$ is large enough
then PCG will converge in $O(1)$ iterations (for a fixed convergence threshold).

\begin{thm}
  \label{thm:main}
  Let $\matK$ be the kernel matrix associated with the $q$-degree homogeneous polynomial kernel
  $k(\x,\z)=(\x^\T \z)^q$. Let $s_\lambda(\matK)\equiv\Trace{(\matK + \lambda \matI_n)^{-1} \matK}$.
  Let $\z_i = \varphi(\x_i)\in \R^s, i=1,\dots,n$ where $\varphi$ is a
  \noun{TensorSketch} map, and $\matZ \in \R^{n \times s}$ have rows $\z^\T_1,\dots,\z^\T_n$.
  Provided that
  \begin{equation}\label{eq:sbound}
    s \geq 4(2+3^q)s_\lambda(\matK)^2/\delta\,,
  \end{equation}
  with probability of at least $1-\delta$, after
  $$
  T = \left\lceil \frac{\sqrt{3}}{2} \ln(2/\epsilon) \right\rceil
  $$
  iterations of PCG on $\matK + \lambda \matI_n$ starting from the all-zeros vector with $\matZ \matZ^\T + \lambda \matI_n$ as a preconditioner
  we will find a $\tilde{\c}$ such that
  \begin{equation}\label{eq:bound}
  \XNorm{\tilde{\c} - \c}{\matK + \lambda \matI_n} \leq \epsilon \XNorm{\c}{\matK + \lambda \matI_n}\,.
  \end{equation}
  In the above, $\c$ is the exact solution of the linear equation at hand (Equation~\ref{eq:krr}).
\end{thm}
\begin{rem*}
  The result is
  stated for the homogeneous polynomial kernel $k(\x,\z)=(\x^\T \z)^q$, but it can be easily generalized to
  the non-homogeneous case $k(\x,\z)=(\x^\T \z + c)^q$ by adding a constant feature to each training point.
\end{rem*}
\begin{rem*}
While  the theorem gives an explicit formula for the number of iterations, we do not recommend to actually
use this formula, and recommend instead the use of standard stopping criteria to declare convergence (these usually involve determining
that the residual norm has dropped below some tolerance). The reason is that the iteration bound holds only with high
probability. On the other hand, a higher probability bound bound holds when we consider a higher bound on the number of iterations.
Thus, using standard stopping criteria renders the algorithm more robust. Furthermore, the bound holds only under exact arithmetic, while in practice PCG is used
with inexact arithmetic.
\end{rem*}

The quantity $s_\lambda(\matK)$ is often referred to as the {\em statistical dimension} or
{\em effective degrees-of-freedom} of the problem. It frequently features in the analysis
of kernel regression~\cite{Zhang05,BK10}, low-rank approximations of kernel matrices~\cite{Bach13},
and analysis of sketching based approximate kernel learning~\cite{AlaouiMahoney15, YPM15}.

Estimating the statistical dimension is a non-trivial task that is outside the scope of this paper
(and can sometimes be avoided: see Section~\ref{sec good Z}).
Nevertheless, the bound does establish that the number of random features required for a constant number
of iterations depends on the statistical dimension, which is always smaller than the number of training points.
Since the kernel matrices often display quick decay in eigenvalues, it can be substantially smaller.
In particular, the number of random features can be $o(n)$
when the statistical dimension is $o(n^{1/2})$.
A discussion on how the statistical dimension
behaves with regard to the training size is outside the scope of this paper. We refer the reader to a recent
discussion by Bach on the subject~\cite{Bach13}.


Before proving the theorem we state some auxiliary definitions and lemmas.
\begin{defn}
  Let $\matA \in \R^{m \times n}$ with $n \geq m$ and let $\lambda \geq 0$. $\matA = \matL \matQ$ is a $\lambda$-LQ factorization of $\matA$ if $\matQ$ is full rank, $\matL$ is lower triangular and $\matL \matL^\T = \matA \matA^\T + \lambda \matI_m$.
\end{defn}

A $\lambda$-LQ factorization always exists, and $\matL$ is invertible for $\lambda > 0$. $\matQ$ has orthonormal rows for $\lambda = 0$. The following proposition refines the last statement a bit.

\begin{prop}
  \label{prop:id}
  If $\matA = \matL \matQ$ is an $\lambda$-LQ factorization $\matA\in\R^{m\times n}$, then $\matQ \matQ^\T + \lambda \matL^{-1} \matL^{-\T} = \matI_m$.
\end{prop}
\begin{proof}
  Multiply $\matL \matL^\T = \matA \matA^\T + \lambda \matI_m$ from the left by $\matL^{-1}$ and from the right by $\matL^{-\T}$ to obtain the equality.
\end{proof}

\begin{prop}
  \label{prop:stat_dim}
  If $\matA = \matL \matQ$ is a $\lambda$-LQ factorization of $\matA\in\R^{m\times n}$, then
  \[
  \FNormS{\matQ}=s_\lambda(\matA \matA^\T)\equiv \Trace{(\matA \matA^\T + \lambda I_m)^{-1} \matA\matA^\T}.
  \]
\end{prop}
\begin{proof}
  Let $\sigma_1,\dots,\sigma_m$ be the singular values of $\matA$.
  \Ken{Changed from $\sigma_1^2$ etc.}
  \begin{eqnarray*}
    \FNormS{\matQ} = \Trace{\matQ \matQ^\T} & = & \Trace{\matI_m  - \lambda \matL^{-1} \matL^{-\T}} \\
    & = & m - \lambda \Trace{\matL^{-1} \matL^{-\T}} \\
    & = & m - \lambda \Trace{(\matA \matA^\T + \matI_m)^{-1}} \\
    & = & m - \sum^{m}_{i=1}\frac{\lambda}{\sigma^2_i + \lambda} \\
    & = & \sum^{m}_{i=1}\frac{\sigma^2_i}{\sigma^2_i + \lambda} \\
    & = &  \Trace{(\matA \matA^\T + \lambda I_m)^{-1}  \matA\matA^\T}
  \end{eqnarray*}
\end{proof}

In addition, we need the following lemma.
\begin{lem}[\cite{ANW14}]
\label{lem:ts_ose}
Let $\matS\in \R^{s\times d^q}$ be a {\sc TensorSketch} matrix, and suppose that $\matA$ and $\matB$ are matrices
with $d^q$ columns. For $s \geq (2 + 3^q)/(\nu^2 \delta)$ we have
\begin{equation*}
\Pr[\FNormS{\matA \matS^\T \matS \matB^\T - \matA \matB^\T} \le \nu^2 \FNormS{\matA}\FNormS{\matB}] \ge 1-\delta\,.
\end{equation*}
\end{lem}

We can now prove Theorem~\ref{thm:main}.
\begin{proof}[Proof of Theorem~\ref{thm:main}]
  We prove that with probability of at least $1-\delta$
  \begin{equation}
    \label{eq:goal}
  \frac{2}{3}(\matZ \matZ^\T + \lambda \matI_n) \preceq \matK + \lambda \matI_n \preceq 2(\matZ \matZ^\T + \lambda \matI_n)\,.
  \end{equation}
  Thus, with probability of $1-\delta$ the relevant condition number is bounded by $3$. For PCG, if the condition number
  is bounded by $\kappa$, we are guaranteed to reduce the error (measured in the matrix norm of the linear equation) to an
  $\epsilon$ fraction of the initial guess after $\lceil \sqrt{\kappa}\ln(2/\epsilon)/2 \rceil$
  iterations~\cite{Shewchuk94}. This immediately leads to the bound in the theorem statement.

  Let $\matV_q\in \R^{n\times d^q}$ be the matrix whose row $i$ corresponds to expanding $\x_i$
  to the values of all possible $q$-degree monomials\footnote{The letter $\matV$ alludes to the
    fact that $\matV_q$ can be thought of as a multivariate analogue of the Vandermonde matrix.},
i.e. $v_q(\x_i)$ in the terminology of Section~\ref{sec:rfm}. We
  have $\matK = \matV_q \matV^\T_q$. Furthermore, there exists a matrix $\matS \in \R^{s \times d^q}$
  such that $\matZ = \matV_q \matS^\T$, so~\eqref{eq:goal} translates to
  \begin{equation*}
  \frac{2}{3}(\matV_q \matS^\T \matS \matV^\T_q + \lambda \matI_n) \preceq \matV_q \matV^\T_q + \lambda \matI_n \preceq 2(\matV_q \matS^\T \matS \matV^\T_q + \lambda \matI_n)\,,
  \end{equation*}
  or equivalently,
  \begin{equation}
    \label{eq:goal1}
  \frac{1}{2}(\matV_q \matV^\T_q + \lambda \matI_n) \preceq \matV_q \matS^\T \matS \matV^\T_q + \lambda \matI_n \preceq \frac{3}{2}(\matV_q \matV^\T_q + \lambda \matI_n)\,.
  \end{equation}

  Let $\matV_q = \matL \matQ$ be a $\lambda$-LQ factorization of $\matV_q$.
  It is well known that for $\matC$ that is square and invertible, $\matA \preceq \matB$ if and only if
  $\matC^{-1} \matA \matC^{-\T} \preceq \matC^{-1} \matB \matC^{-\T}$. Applying this to the previous equation
  with $\matC = \matL$ implies that~\eqref{eq:goal1} holds if and only if
  \begin{equation}
    \label{eq:goal1.5}
  \frac{1}{2}\matI_n \preceq \matQ \matS^\T \matS \matQ^\T + \lambda \matL^{-1} \matL^{-\T} \preceq \frac{3}{2}\matI_n\,.
  \end{equation}

  A sufficient condition for~\eqref{eq:goal1.5} to hold is that
  \begin{equation}
    \label{eq:goal2}
    \TNorm{\matQ \matS^\T \matS \matQ^\T + \lambda \matL^{-1} \matL^{-\T} - \matI_n} \leq \frac{1}{2}\,.
  \end{equation}
  According to Proposition~\ref{prop:id} we have
  \begin{equation*}
  \TNorm{\matQ \matS^\T \matS \matQ^\T + \lambda \matL^{-1} \matL^{-\T} - \matI_n} = \TNorm{\matQ \matS^\T \matS \matQ^\T - \matQ \matQ^\T}\,.
  \end{equation*}
  According to Lemma~\ref{lem:ts_ose}, if $s \geq 4(2+3^q)\FNorm{\matQ}^4/\delta$,  then with probability of at least $1-\delta$ we have
  \begin{equation*}
  \TNorm{\matQ \matS^\T \matS \matQ^\T - \matQ \matQ^\T} \leq \FNorm{\matQ \matS^\T \matS \matQ - \matQ \matQ^\T} \leq \frac{1}{2}\,.
  \end{equation*}
  Now complete the proof using the equality
  $\FNormS{\matQ} = s_\lambda(\matV_q\matV_q^\T) = s_\lambda(\matK)$ (Proposition~\ref{prop:stat_dim}).
\end{proof}

\subsection{Other Kernels and Feature Maps}

Close inspection of the proof reveals that the crucial ingredient is the matrix multiplication lemma
(Lemma~\ref{lem:ts_ose}). In the following, we generalize Theorem~\ref{thm:main} to feature maps
which have similar structural properties. The proof, which is mostly analogous to the proof
of Theorem~\ref{thm:main}, is included in the Appendix.

In the following, for finite ordered sets ${\cal U}, {\cal V} \subset {\cal H}_k$ we denote by $\matK(\cal{U},\cal{V})$
the Gram matrix associated with this two sets, i.e. $\matK_{ij} = \dotprod{\u_i}{\v_j}{{\cal H}_k}$ for $\u_i \in \cal{U}$ and $\v_j \in \cal{V}$. We assume that the feature map defines a transformation from ${\cal H}_k$ to $\R^s$,
that is, a sample from the distribution is a  function $\varphi : {\cal H}_k \to \R^s$. A feature map is linear
if every sample $\varphi$ it can take is linear. \Ken{this last sentence is not clear to me} \Haim{Edited.}

\begin{defn}
  A linear feature map has an {\em approximate multiplication property} with $f(\nu, \delta)$ if
  $\varphi$ with at least $f(\nu, \xi, \delta)$ random features has that for all finite ordered
  sets ${\cal U}, {\cal V} \subset {\cal H}_k$ the following holds with probability of at least $1-\delta$:
  \begin{equation*}
  \FNormS{\matZ_{\cal U} \matZ^\T_{\cal V} - \matK({\cal U},{\cal V})} \le \nu^2 \Trace{\matK({\cal U},{\cal U})}\Trace{\matK({\cal V},{\cal V})} + \xi^2
  \end{equation*}
   where $\matZ_{\cal U}$ (resp. $\matZ_{\cal V}$) is the matrix whose row
  $i$ corresponds to applying $\varphi$ to $\u_i$ (resp. $\v_i$).
\end{defn}

\begin{thm}
  \label{thm:genmain}
  Suppose that $\varphi$ is a sample from a  feature map that has an approximate multiplication property with $f(\nu, \delta)$. Suppose that $\varphi$ has
  $s \geq  f(s_\lambda(\matK)^{-1}/2, 0, \delta)$ or  $s \geq  f(s_\lambda(\matK)^{-1}/2\sqrt{2}, 1/2\sqrt{2}, \delta)$ features.
  Let $\z_i = \varphi(k(\x_i, \cdot))\in \R^s, i=1,\dots,n$ and $\matZ \in \R^{n \times s}$ have rows $\z^\T_1,\dots,\z^\T_n$.
  With probability of at least $1-\delta$, after
  $$
  T = \left\lceil \frac{\sqrt{3}}{2} \ln(2/\epsilon) \right\rceil
  $$
  iterations of PCG on $\matK + \lambda \matI_n$ starting from the all-zeros vector with $\matZ \matZ^\T + \lambda \matI_n$ as a preconditioner we will find a $\tilde{\c}$ such that
  \begin{equation*}
  \XNorm{\tilde{\c} - \c}{\matK + \lambda \matI_n} \leq \epsilon \XNorm{\c}{\matK + \lambda \matI_n}\,.
  \end{equation*}
\end{thm}

Currently, there is no proof that the approximate multiplication property holds for any feature map except
for {\sc TensorSketch}.

\section{Multiple Level Sketching}

We now show that the dependence on the statistical dimension can be improved by composing
multiple sketching transforms. The crucial observation is that after the initial random feature
transform, the training set is embedded in a Euclidean space. This suggests the composition of
well-known transforms such as the Subsampled Randomized Hadamard Transform (SRHT) and the Johnson-Lindenstrauss
transform with the initial random feature map. A similar idea, referred to as {\em compact random
features}, was explored in the context of the random features method by Hamid et al.~\cite{HXGD14}.

First, we consider the use of the SRHT after the initial {\sc TensorSketch}
for the polynomial kernel. To that end we recall the definition of the SRHT.
Let $m$ be a power of 2. The $m \times m$ matrix of the
Walsh-Hadamard Transform (WHT) is defined recursively as,
\begin{equation*}
 \matH_m = \left[
\begin{array}{cc}
  \matH_{m/2} &  \matH_{m/2} \\
  \matH_{m/2} & -\matH_{m/2}
\end{array}\right],
\ \mbox{with} \
\matH_2 = \left[
\begin{array}{cc}
  +1 & +1 \\
  +1 & -1
\end{array}\right].
\end{equation*}
\begin{defn}
\label{def:srht}
Let $m$ be some integer which is a power of 2, and $s$ an integer. A \emph{Subsampled Randomized Walsh-Hadamard Transform (SRHT)} is an $s \times m$ matrix of the form
\begin{equation*}
\matS =  \frac{1}{\sqrt{s}} \matP \matH  \matD\,
\end{equation*}
where $\matD$ is a random diagonal matrix of size $m$ whose entries are independent random signs,
and $\matH$ is a Walsh-Hadamard matrix of size $m$, and $\matP$ is a random sampling matrix.
\end{defn}
The recursive nature of the WHT matrix allows for a quick multiplication of a SRHT matrix by a vector.
In particular, if $\matS \in \R^{s\times m}$ is an SRHT, then $\matS \x$ can be computed in $O(m \log(s))$~\cite{AilonLiberty08}.

The proposed algorithm proceeds as follows. First, we apply a random feature transform with $s_1$ features to
obtain $\matZ_1 \in \R^{n\times s_1}$. We now apply a SRHT $\matS_2 \in \R^{s_2 \times s_1}$ to the rows
of $\matZ_1$, that is compute $\matZ_2 = \matZ_1 \matS^\T_2$, and use $\matZ_2\matZ^\T_2 + \lambda \matI_n$
as a preconditioner for $\matK + \lambda \matI_n$.

The following theorem establishes that  using this scheme it is possible for the polynomial kernel
to construct a good preconditioner with only $O(s_\lambda(\matK) \log(s_\lambda(\matK)))$ columns
in $\matZ_2$.
\begin{thm}
  \label{thm:main2}
  Let $\matK$ be the kernel matrix associated with the $q$-degree homogeneous polynomial kernel
  $k(\x,\z)=(\x^\T \z)^q$. Let $s_\lambda(\matK)\equiv\Trace{(\matK + \lambda \matI_n)^{-1} \matK}$.
  Assume that $\TNorm{\matK} \geq \lambda$.
  Let $\varphi : \R^d \to \R^{s_1}$ be a {\sc TensorSketch} mapping with
  \begin{equation*}
  s_1 \geq 32(2+3^q)s_\lambda(\matK)^2/\delta\,,
  \end{equation*}
  and $\matS_2 \in \R^{s_2 \times s_1}$ be a SRHT with $s_2 = \Omega(s_\lambda(\matK) \log(\s_\lambda(\matK))/\delta^2)$.
  Let $\z_i = \matS_2 \varphi(\x_i)\in \R^s, i=1,\dots,n$ and $\matZ_2 \in \R^{n \times s}$
  have rows $\z^\T_1,\dots,\z^\T_n$.
  Then, with probability of at least $1-\delta$,
 after
  $$
  T = \left\lceil \frac{\sqrt{3}}{2} \ln(2/\epsilon) \right\rceil
  $$
  iterations of PCG on $\matK + \lambda \matI_n$ starting from the all-zeros vector with $\matZ \matZ^\T + \lambda \matI_n$ as a preconditioner
  we will find a $\tilde{\c}$ such that
  \begin{equation}\label{eq:bound2}
  \XNorm{\tilde{\c} - \c}{\matK + \lambda \matI_n} \leq \epsilon \XNorm{\c}{\matK + \lambda \matI_n}\,.
  \end{equation}
  In the above, $\c$ is the exact solution of the linear equation at hand (Equation~\ref{eq:krr}).
\end{thm}

\begin{proof}
  Let $\matV_q$ and $\matQ$ be as defined in the proof of Theorem~\ref{thm:main}.
  Let $\matS_1$ be the matrix that corresponds to $\varphi$,
  and let $\matS = \matS_2 \matS_1$. We have $\matZ = \matV_q \matS^\T$. Following the proof of Theorem~\ref{thm:main}
  it suffices to prove that
  \begin{equation*}
  \TNorm{\matQ \matS^\T \matS \matQ^\T - \matQ \matQ^\T} \leq \frac{1}{2}\,.
  \end{equation*}

  Noticing that
  \begin{equation*}
    \TNorm{\matQ \matS^\T \matS \matQ^\T - \matQ \matQ^\T} \leq \TNorm{\matQ \matS_1^\T \matS^\T_2\matS_2 \matS_1 \matQ^\T - \matQ \matS_1^\T \matS_1 \matQ^\T} + \TNorm{\matQ \matS_1^\T \matS_1 \matQ^\T - \matQ \matQ^\T}
  \end{equation*}
  it suffices to prove that each of the two terms is bounded by $\frac{1}{4}$ with probability $1-\delta/2$.
  Due to the lower bound on $s_1$ this holds for the right term as explained in the proof of Theorem~\ref{thm:main}.

  For the left term, we used recent results by Cohen et al.~\cite{CNW15} that show that for a
  matrix $\matA$ and a SRHT $\Pi$ with $O(\tilde{r}(\matA)\log(\tilde{r}(\matA))/\nu^2)$ rows,
  where $\tilde{r}(\matA)\equiv \FNormS{\matA} / \TNormS{\matA}$ is the {\em stable rank} of $\mat{A}$,
  we have
  \begin{equation*}
  \TNorm{\matA^\T \Pi^\T \Pi \matA - \matA^\T \matA} \leq \nu \TNormS{\matA}
  \end{equation*}
  with high probability.

  We bound the left term by applying this result to $\matA = \matS_1 \matQ^\T$. First, we note that
  $\TNormS{\matQ^\T}=\TNorm{\matK}/(\TNorm{\matK} + \lambda)\geq 1/2$. Next, we note that
  \begin{eqnarray*}
    \TNorm{ \matS_1\matQ^\T} & = &\TNorm{\matQ \matS^\T_1\matS_1\matQ^\T}^{1/2}\\
    & \leq &\TNorm{\matQ^\T} + \TNorm{\matQ \matS^\T_1\matS_1\matQ^\T - \matQ \matQ^\T}^{1/2}\,.
  \end{eqnarray*}
  The second term has already been proven to be $O(1)$, while the first is bounded by one.
  Therefore, we conclude that $\TNorm{\matS_1\matQ^\T} = O(1)$ so  $O(\tilde{r}(\matS_1\matQ^\T)\log(\tilde{r}(\matS_1\matQ^\T)))$
  row is sufficient. In addition, since $\Expect{\FNormS{\matS_1\matQ^\T}}=\FNormS{\matQ^\T}$ we
  conclude that with high probability $\FNormS{\matS_1\matQ^\T} = O(\FNormS{\matQ^\T}) = O(s_\lambda(\matK))$. Finally, recall
  that $\TNormS{\matQ^T}\geq 1/2$ so $\tilde{r}(\matS_1\matQ^\T) = O(s_\lambda(\matK))$.
\end{proof}

\begin{rem*}
Notice that for multiple level sketching we require that $\TNorm{\matK} \geq \lambda$. This assumption
is very reasonable: from a statistical point of view it should be the case that $\lambda \ll \TNorm{\matK}$
(otherwise, the regularizer dominates the data). From a computational point of view, if
$\lambda = \Omega(\TNorm{\matK})$ then the condition number of $\matK + \lambda \matI_n$ is $O(1)$ to begin
with (although the constant might be huge). In particular, if $\lambda \geq \TNorm{\matK}$ then the
condition number is bounded by $2$.
\end{rem*}

\begin{rem*}
The dependence on $\delta$ in the previous theorem (and also in Theorem~\ref{thm:main}) can be improved to $O(\poly{\log(1/\delta)})$
by considering a fixed failure probability and repeating the algorithm $O(\log(1/\delta))$ times.
If $\tilde{\c}_1, \ldots, \tilde{\c}_r$ are the candidate solutions, for $r = O(\log(1/\delta))$, then a sufficiently good solution can be found by looking at the differences $\XNorm{\tilde{\c}_i - \tilde{\c}_j}{\matK + \lambda \matI_n}$ for all $i$ and $j$, and outputting an $i$ for which, if we sort these differences for all $j$, the median value of the sorted list for $i$ is the smallest. The analysis is based on Chernoff bounds and the triangle inequality, and requires adjusting $\epsilon$ by a constant factor. We omit the details.
\end{rem*}

From a computational complexity point of view, in many cases one can set $s_1$ to be very large without increasing
the asymptotic cost of the algorithm. For example, for random Fourier features, if $\sum^n_{i=1}\nnz{\x_i} = O(n\log{s_2})$
then it is possible to set $s_1 = \Theta(n)$ without increasing the asymptotic cost of the algorithm.

\begin{algorithm}[t]
\begin{algorithmic}[1]

  \STATE \textbf{Input: }Data $(\x_1,y_1),\dots, (\x_n,y_n)\in \R^d \times \R$, kernel $k(\cdot, \cdot)$, feature map
  generating algorithm $T$, $\lambda > 0$, $s_3 < s_2 < s_1 < n$, accuracy parameter $\epsilon > 0$.

  \STATE \Ken{This line intentionally left blank} \Haim{yes}

  \STATE Compute kernel matrix $\matK \in \R^{n \times n}$.

  \STATE Using $T$, create a feature map $\varphi : \R^d \to \R^{s_1}$.

  \STATE Compute $\z_i = \varphi(\x_i), i=1,\dots,n$ and stack them in a
  a matrix $\matZ_1$.

  \STATE Apply a subsampled randomized Hadamard transform $\matS_2 \in \R^{s_2 \times s_1}$
  to the rows of $\matZ_1$ to form $\matZ_2 = \matZ_1 \matS^\T_2$.

  \STATE Let $\matS_3 \in \R^{s_3 \times s_2}$ be a random matrix whose entries independent standard
  normal random variables. Compute $\matZ_3 = \frac{1}{\sqrt{s_3}}\matZ_2 \matS^\T_3$.

  \STATE Solve $\matL^{\T} \matU = \matZ^\T$ for $\matU$.

  \STATE Stack $y_1,\dots,y_n$ in a vector $\y\in\R^n$.

  \STATE Solve $(\matK + \lambda \matI_n)\c = \y$ using PCG to accuracy $\epsilon$ with
  $\matZ \matZ^\T + \lambda \matI_n$ as a preconditioner. In each iteration, in order to apply
  the preconditioner to some $\x$, compute $\lambda^{-1}(\x - \matU^\T \matU \x)$.

  \STATE

  \STATE {\bf return} $\tilde{\c}$

\end{algorithmic}

\protect\caption{\label{alg:x2fastkrr}Multiple Level Sketching Algorithm.}
\end{algorithm}

It is also possible to further reduce the dimension to $O(s_\lambda(\matK))$  using a dense random projection,
i.e. multiplying
$\matZ_2$ from the right by a scaled random matrix with subgaussian entries. We omit the proof as it is analogous to the
proof of Theorem~\ref{thm:main2}. Pseudo-code description of the three-level sketching algorithm is given as
Algorithm~\ref{alg:x2fastkrr}.

\section{Adaptively Setting the Sketch Size}\label{sec good Z}


It is also possible to adaptively set $s$. The basic idea is to successively form larger $\matZ$'s until we have one that is good enough.
The following theorem implies quickly testable conditions for this, again for the polynomial kernel and {\sc TensorSketch}.
\begin{thm}
  \label{thm:testprecond}
Let $\matP\in\R^{n\times n}$ be the orthogonal projection matrix on the subspace spanned by the eigenvectors
of $\matZ \matZ^\T$  whose corresponding eigenvalues are bigger than $0.05 \lambda$,
and let $\matV_q$ have $\matK = \matV_q\matV_q^\T$.
Suppose that:
\begin{enumerate}
\item $\TNorm{(\matI_n - \matP) \matK (\matI_n - \matP)} \leq 0.1\lambda$.
\item For all $\x$, $\TNormS{\matV_q^\T \matP \x} = (1 \pm 0.1) \TNormS{\matZ^\T \matP \x}$.
\end{enumerate}
Then the guarantees of Theorem~\ref{thm:main} hold.
Moreover, if $s \geq C s_\lambda(\matK)^2/\delta $ for some sufficiently large constant $C$,
then with probability of at least $1-\delta$ conditions 1 and 2 above hold as well.
\end{thm}

\begin{proof}
  We show that there exists constants $m$ and $M$ such that
  \begin{equation*}
  m (\matZ \matZ^\T + \lambda \matI_n) \preceq \matK + \lambda \matI_n \preceq M (\matZ \matZ^\T + \lambda \matI_n)\,.
  \end{equation*}
  This hold if there is a constant $c < 1$ such that for all $\y$
  \begin{equation*}
  \y^T(\matK + \lambda \matI_n)\y = (1 \pm c)\y^\T(\matZ \matZ^\T + \lambda \matI_n) \y
  \end{equation*}

Without loss of generality we restrict ourselves to $\TNorm{\y}=1$. We have,
\begin{eqnarray*}
\y^\T (\matK+\lambda\matI_n)\y
	   & = & \lambda\TNormS{\y} + \y^\T \matK \y
	\\ & = & \lambda + (\matP \y+(\matI -\matP )\y)^\T \matK (\matP \y+(\matI -\matP )\y)
	\\ & =  & \lambda + \y^\T \matP  \matK  \matP  \y + \y^\T (\matI -\matP )\matK (\matI -\matP )\y + 2\y^\T \matP  \matK  (\matI -\matP )\y
	\\ & =  & (1\pm 0.1)\lambda + (1 \pm 0.1)\y^\T \matP  \matZ \matZ ^\T \matP  \y + 2\y^\T \matP  \matK  (\matI -\matP )\y \pm 0.1\lambda
        \\ & =  & (1\pm 0.1)\lambda + (1 \pm 0.1)\y^\T \matP  \matZ \matZ ^\T \matP  \y + 2\y^\T \matP  \matK  (\matI -\matP )\y
	\\ &\qquad & + [ (1\pm 0.1) \y^\T(\matI -\matP )\matZ \matZ ^\T(\matI -\matP )\y + (1\pm 0.1)\y^\T(\matI -\matP )\matZ \matZ ^\T \matP \y
    \\ & \qquad &  - (1\pm 0.1) \y^\T(\matI -\matP )\matZ \matZ ^\T(\matI -\matP )\y]
        \\ & =  & (1\pm 0.1)\lambda + (1 \pm 0.1)\y^\T \matP  \matZ \matZ ^\T \matP  \y + 2\y^\T \matP  \matK  (\matI -\matP )\y
	\\ &\qquad & + [ (1\pm 0.1) \y^\T(\matI -\matP )\matZ \matZ ^\T(\matI -\matP )\y
    \\ &\qquad & + (1\pm 0.1)\y^\T(\matI -\matP )\matZ \matZ ^\T \matP \y \pm (1\pm 0.1)0.1\lambda]
	\\ & = & (1\pm 0.21)\lambda + (1\pm 0.1)\y^\T \matZ \matZ ^\T \y + 2\y^\T \matP  \matK  (\matI -\matP )\y
	\\ & = & (1\pm 0.21)\y^\T(\matZ \matZ ^\T + \lambda \matI)\y + 2\y^\T \matP  \matK  (\matI -\matP )\y
	\\ & = & (1\pm 0.21)\y^\T(\matZ \matZ ^\T + \lambda \matI)\y \pm 2\TNorm{\y^\T \matP \matV_q } \TNorm{\matV_q^\T(\matI -\matP )\y}
	\\ & = & (1\pm 0.21)\y^\T(\matZ \matZ ^\T + \lambda \matI)\y \pm 2\TNorm{\y^\T \matP \matV_q } \sqrt{0.1\lambda}(1\pm 0.1)
\end{eqnarray*}
The fourth equality is due to condition 2 (which bounds $\y^\T \matP  \matK  \matP  \y$ since $\matK = \matV_q \matV^\T_q$) and condition
1 (which bounds $\y^\T (\matI -\matP )\matK (\matI -\matP )\y$). In the fifth equality we note that
because $\matP$ is a projection on an eigenspace of $\matZ \matZ^\T$ then $(\matI -\matP )\matZ \matZ ^\T \matP = 0$.
In the sixth equality we use the definition of $\matP$ to bound $\y^\T(\matI -\matP )\matZ \matZ ^\T(\matI -\matP )\y \leq 0.1\lambda$.
In the ninth equality we use Cauchy-Schwartz. And in the tenth equality we use condition 1 again.

If $\sqrt{0.1 \lambda} \leq \TNorm{\y^\T \matP \matZ}/3$ then the last term is at most
$\frac23 1.1 \TNormS{\y^\T \matP \matV_q} = \frac{11}{15} \y^\T \matZ \matZ^\T \y$.
Notice that now $(0.21 + 11/15) < 1$, so $\y^\T (\matK+\lambda\matI_n)\y = (1 \pm c) \y^\T(\matZ \matZ ^\T + \lambda \matI)\y$
for some constant smaller than 1.

If $\sqrt{0.1 \lambda} \geq \TNorm{\y^\T \matP \matZ}/3$, then the last term is bounded
$2\TNorm{\y^\T \matP \matV_q } \sqrt{0.1\lambda}(1\pm 0.1) \leq \frac{6}{11}\lambda \leq \frac{6}{11}\y^\T (\matZ\matZ^\T + \lambda \matI_n) \y $
and again $\y^\T (\matK+\lambda\matI_n)\y = (1 \pm c) \y^\T(\matZ \matZ ^\T + \lambda \matI)\y$
for some constant smaller than 1.

As for the other direction, if the conditions of Theorem~\ref{thm:main} hold then
we showed in the proof of Theorem~\ref{thm:main} that Equation~\ref{eq:goal} holds with high probability.
It can be easily seen that by adjusting the constant in front of Equation~\eqref{eq:sbound} we can adjust the constants
in Equation~\eqref{eq:goal} to be as small as needed.
Equation~\ref{eq:goal} immediately implies condition 2. As for condition 1, if it was not true than
there is a unit vector $\y$ such that $\y^\T(\matK + \lambda \matI_n) \y \geq 1.1\lambda$ but
$\y^\T (\matZ \matZ^\T + \lambda \matI_n) \y \leq 1.05 \lambda$. Thus, if the adjusted constant is small enough it is impossible
for (the modified) equation~\eqref{eq:goal} to hold (i.e., $1.1\lambda \leq (1+\epsilon_0)1.05\lambda$ cannot hold for small
enough $\epsilon_0$).
\end{proof}

This theorem can be used in the following way. We start with some small $s$, generate $\matZ$ and test it.
If $\matZ$ is not good enough, we double $s$ and generate a new $\matZ$. We continue until we have a good
enough preconditioner (which will happen when $s$ is large enough per Theorem~\ref{thm:main}).
Testing condition 1 can be accomplished using simple power-iteration.
Testing of condition 2 can be done by constructing a subspace embedding to the range of $\matV^\T_q \matP$ using
{\sc TensorSketch}. We omit the technical details.

\section{Experiments}
\label{sec:experiments}

In this section we report experimental results with an implementation of our proposed
one-level sketching algorithm. To allow the code to scale to datasets of size one million examples
and beyond, we designed our code to use distributed memory parallelism using MPI. For most distributed
matrix operations we rely on the Elemental library~\cite{PoulsonEtAl13}. We experiment on clusters
composed on Amazon Web Services EC2 instances.


In our experiments, we use standard classification and regression datasets, using Regularized Least Squares Classification
(RLSC) to convert the classification problem to a regression problem. We use two kernel-feature map combinations:
the Gaussian kernel $k(\x, \z) = \exp(-\TNormS{\x - \z} /  2\sigma^2)$ along with random Fourier features,
and the Polynomial kernel $k(\x, \z)=(\gamma \x^\T \z + c)^q$ along with {\sc TensorSketch}. We declare convergence
of PCG when the iterate $\c$ is such that $\TNorm{\y - (\matK + \lambda \matI_n)\c} \leq \tau \TNorm{\y}$.
For multiple right hand sides we require this to hold for all right hand sides.
We set $\tau=10^{-3}$ for classification and $\tau=10^{-5}$ for regression; while these are not particularly small,
our experiments did reveal that for the tested datasets further
error reduction does not improve the generalization performance (see also subsection~\ref{sec:additional}).
\Ken{Is the generalization performance significantly worse
if $\tau=1/1000?$ Also: this convergence test is used because it's known a priori that a "near-exact" fit is possible?}
\Haim{It is actually 10e-3, and it is vector-wise. We know there is an exact fit because it is a linear system. The answer to setting tau smaller is in email.}

\subsection{Comparison to the Random Features Method}
\label{sec:vs-rfm}

\begin{figure}
\begin{centering}
\begin{tabular}{cc}
\includegraphics[width=0.3\textwidth]{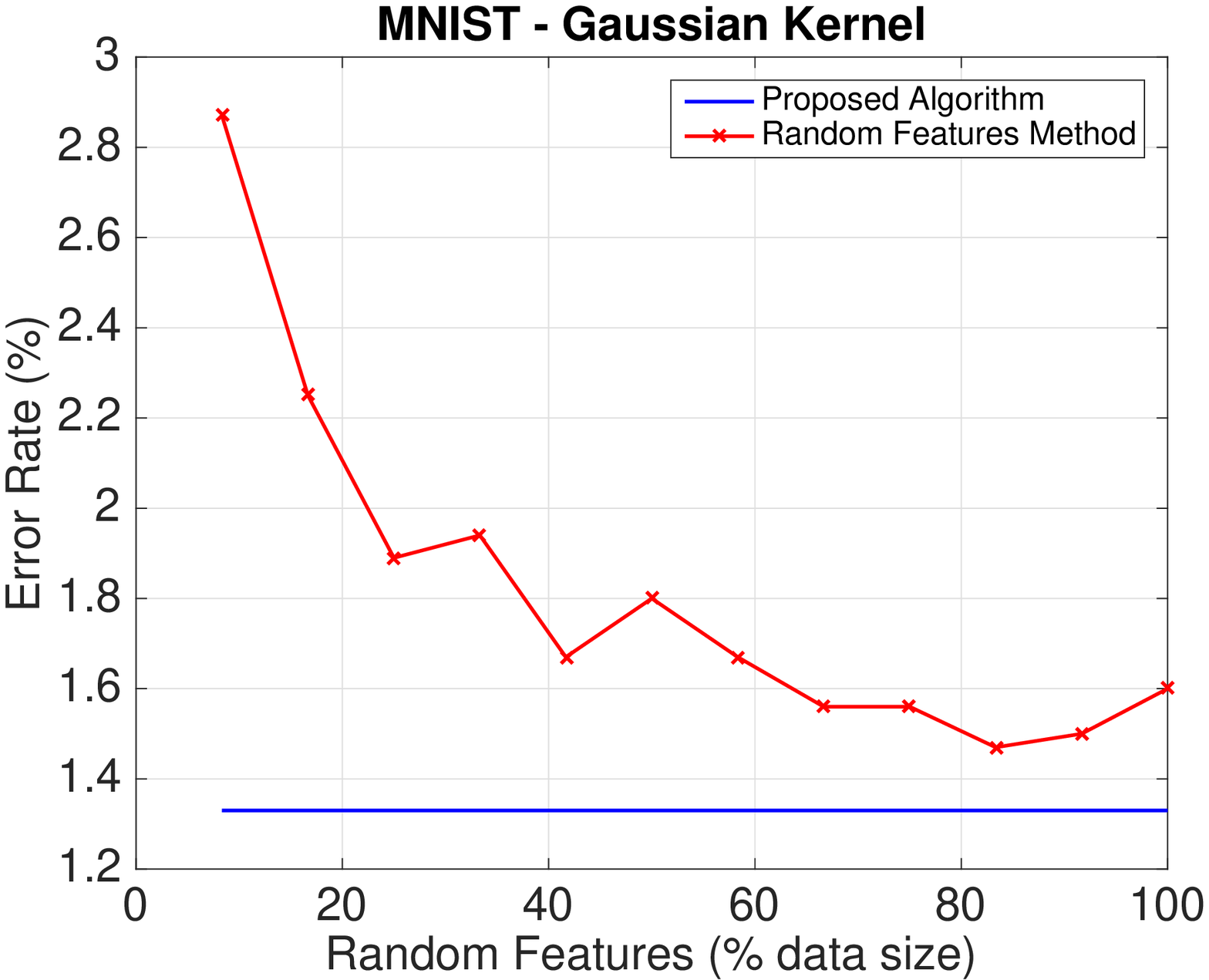} & \includegraphics[width=0.3\textwidth]{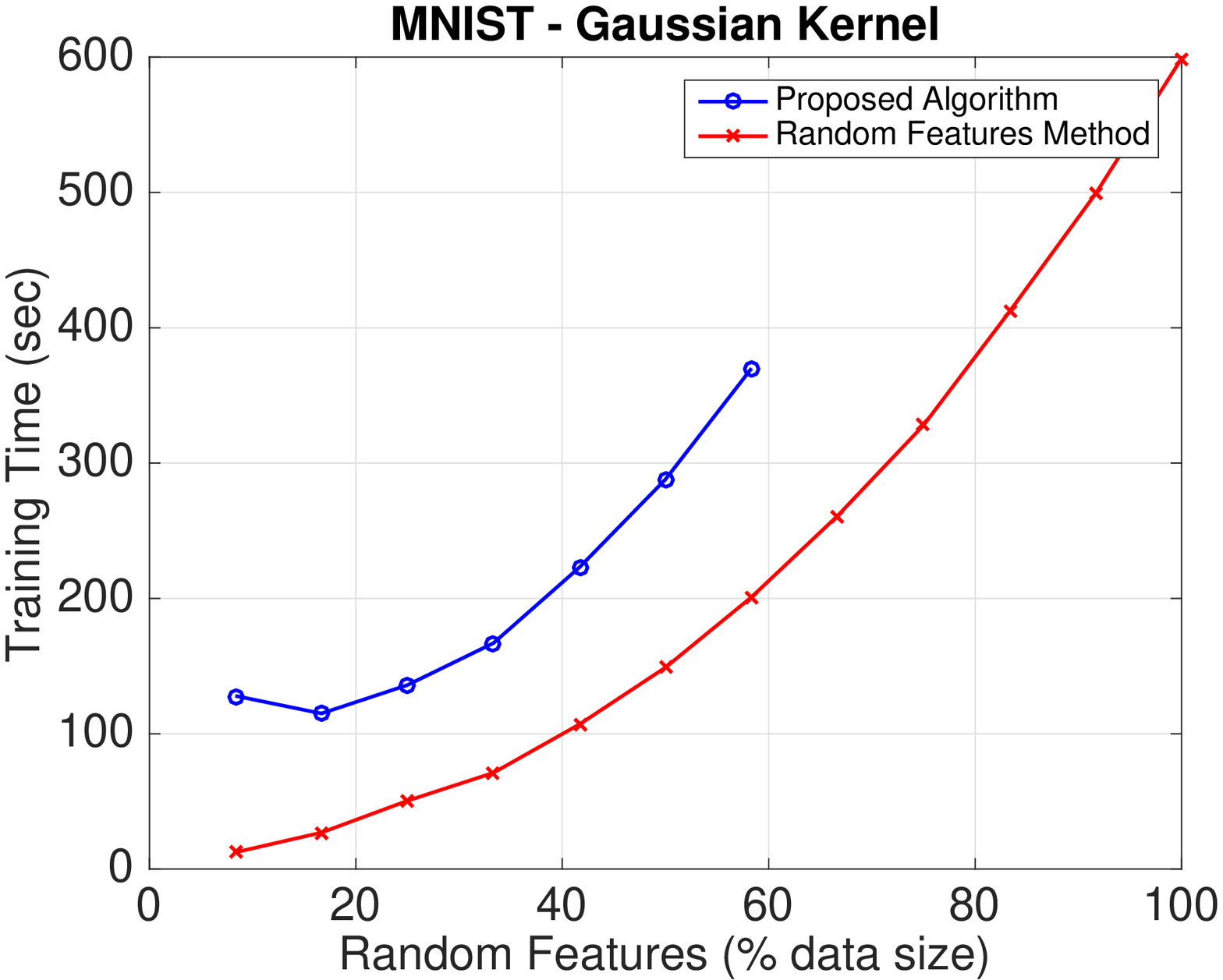}\\
\includegraphics[width=0.3\textwidth]{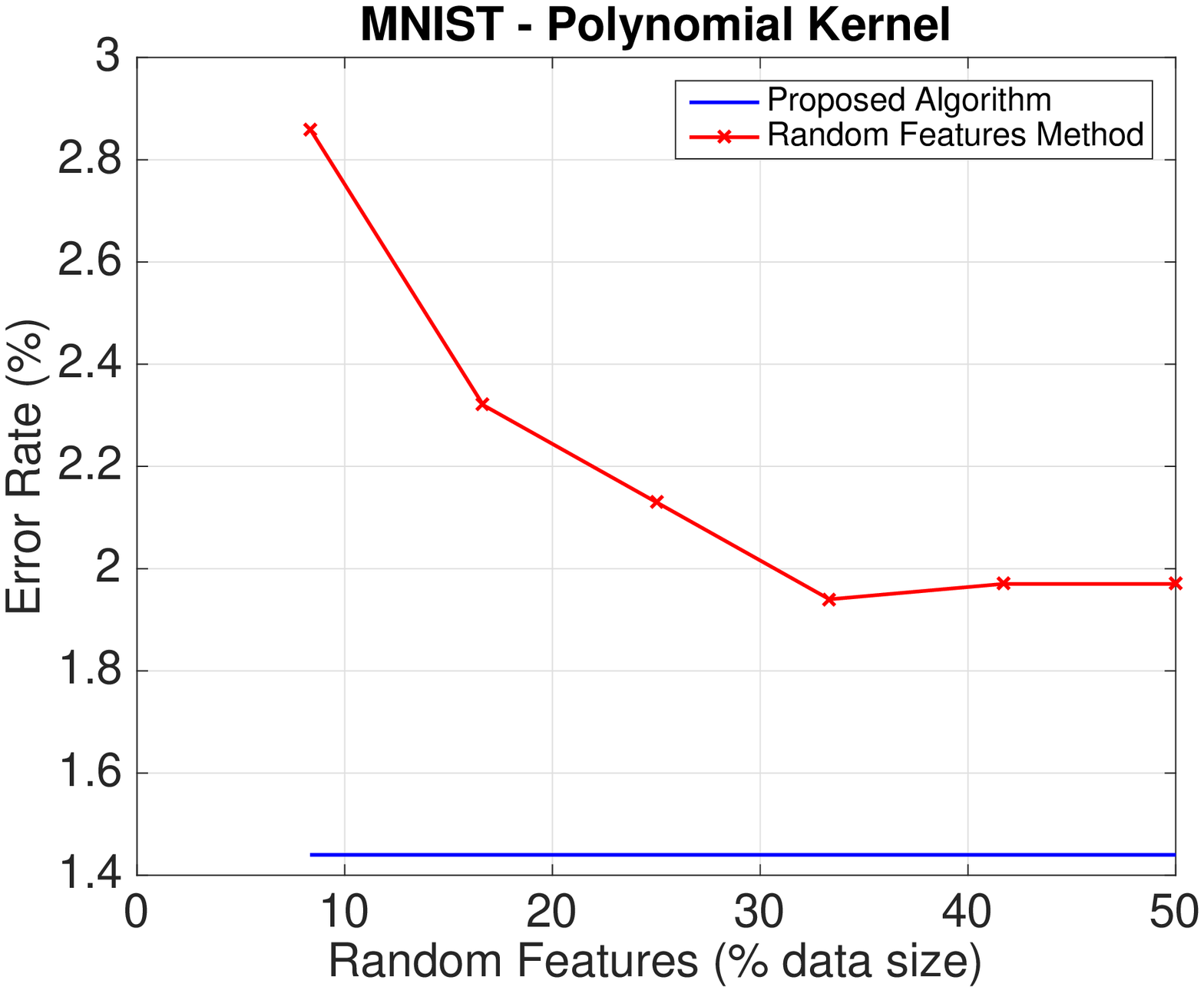} & \includegraphics[width=0.3\textwidth]{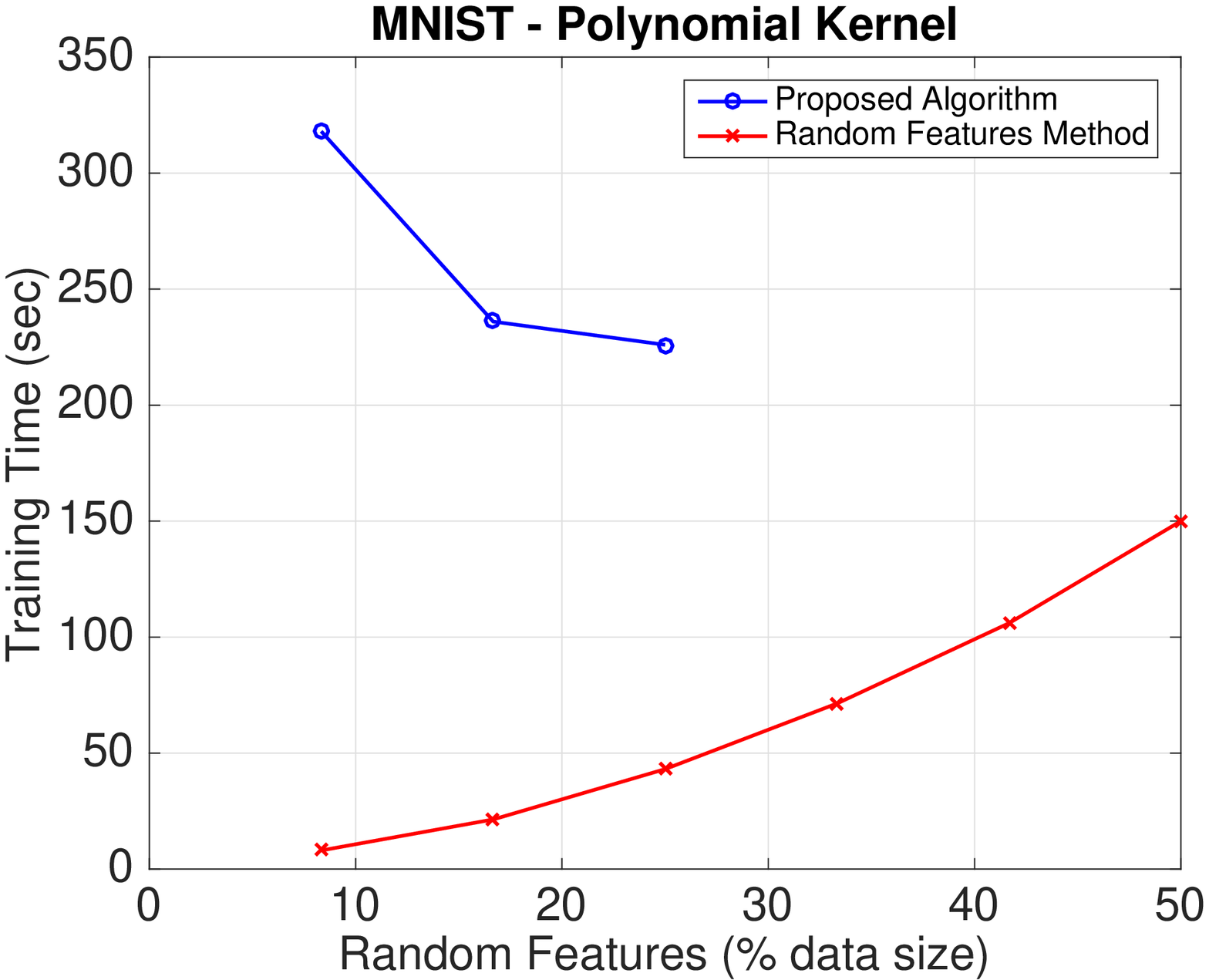}
\end{tabular}
\par\end{centering}
\protect\caption{\label{fig:mnist_vs_rf}Comparison with random features on MNIST. }
\end{figure}

\begin{figure}
\begin{centering}
\begin{tabular}{cc}
\includegraphics[width=0.3\textwidth]{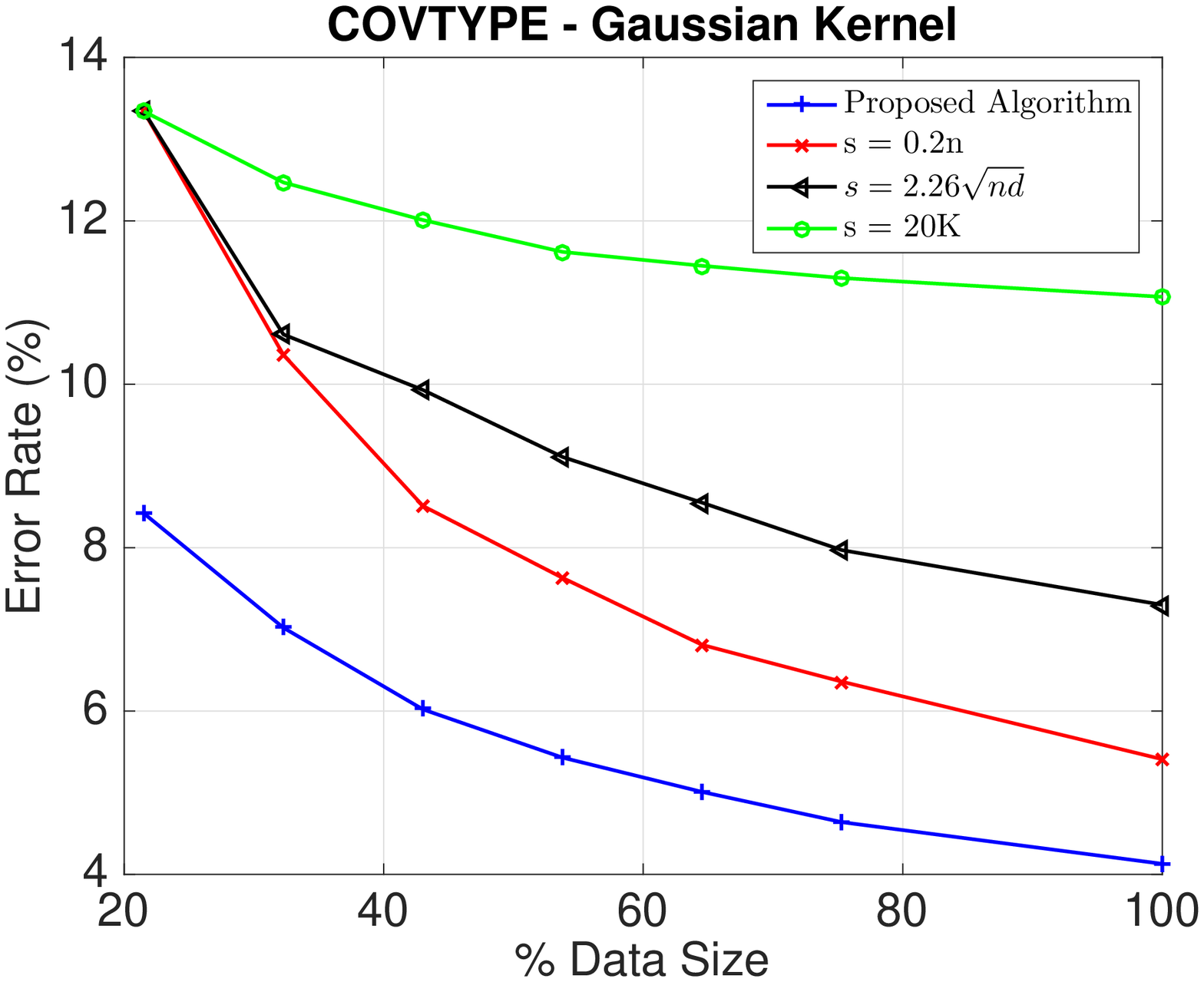} & \includegraphics[width=0.3\textwidth]{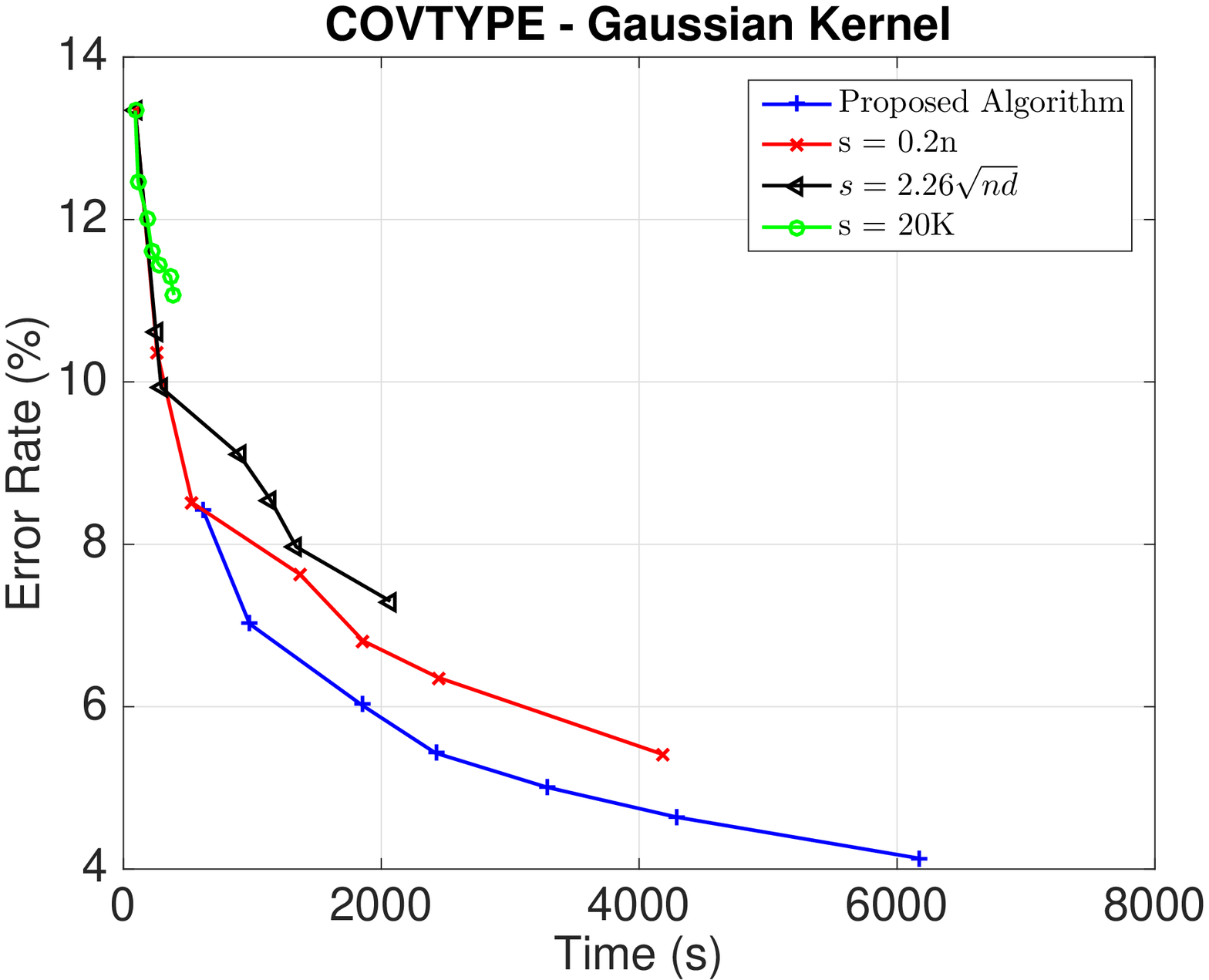}\\
\includegraphics[width=0.3\textwidth]{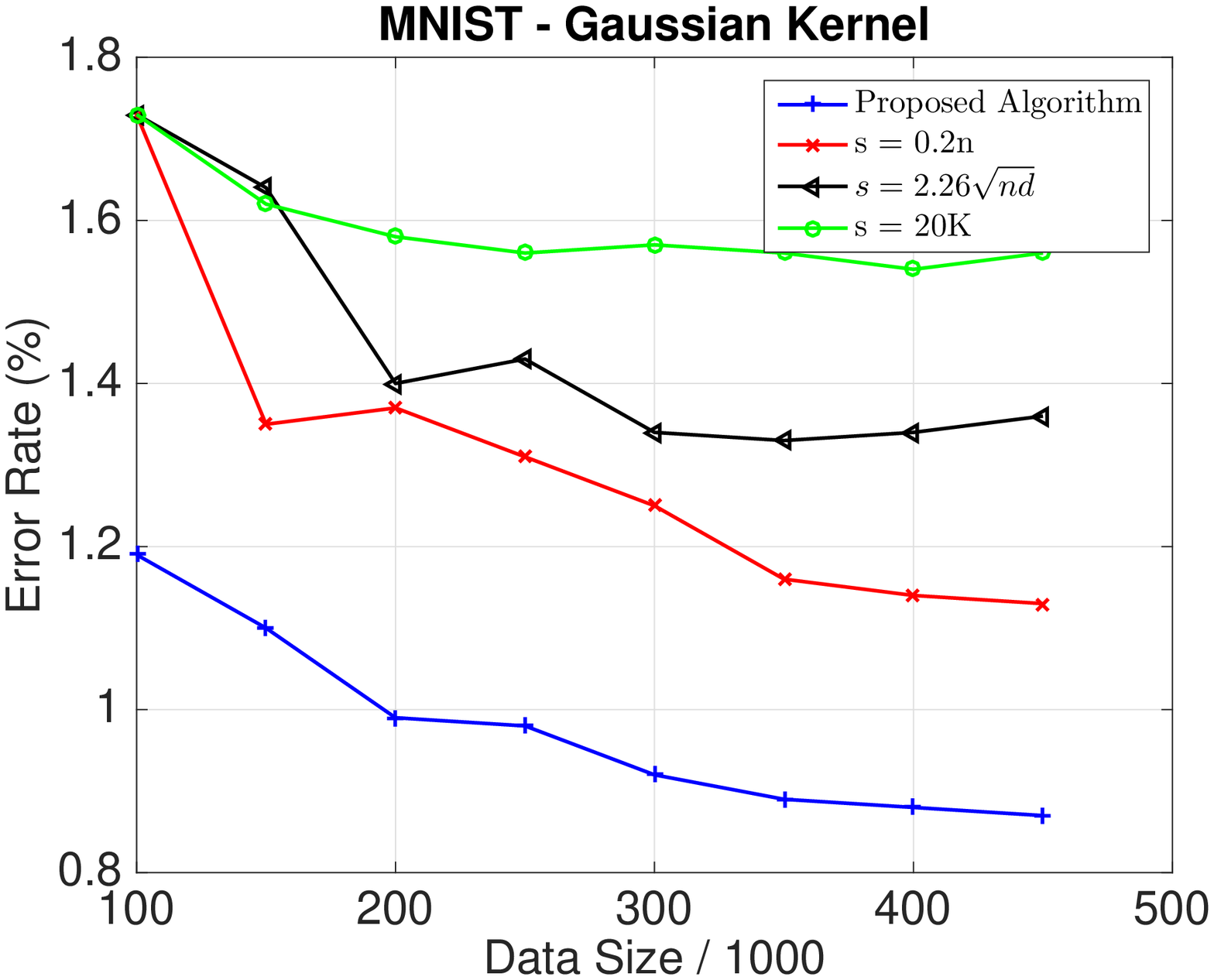} & \includegraphics[width=0.3\textwidth]{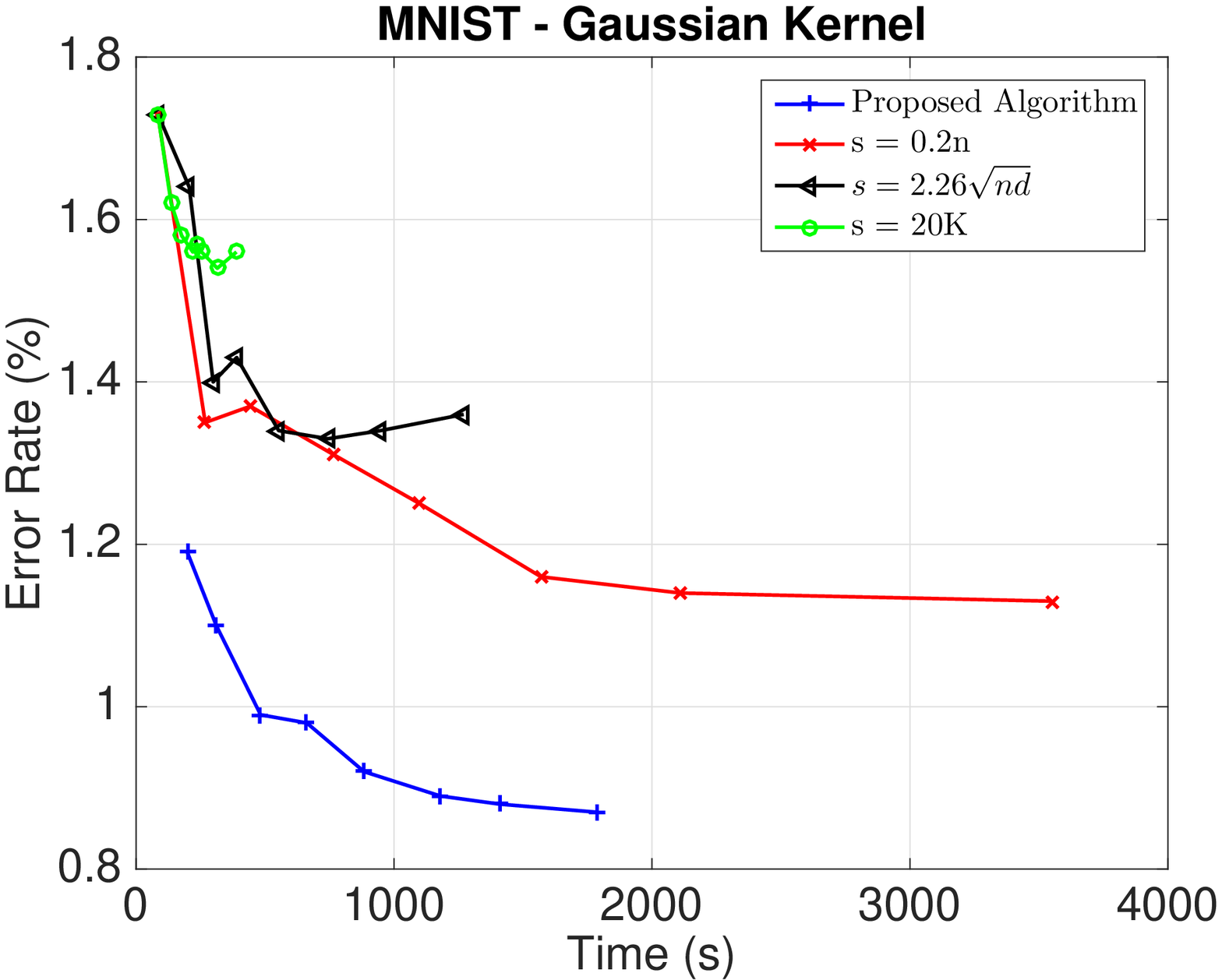}
\end{tabular}
\par\end{centering}
\protect\caption{\label{fig:err_vs}Error as a function of time and data size on MNIST and COVTYPE. }
\end{figure}

In this section we compare our method to the random features method, as defined in Section~\ref{sec:rfm}.
Our goal is to demonstrate
that our method, which solves the non-approximate kernel problem to relatively high accuracy,
is able to fully leverage the data and deliver better generalization results than is possible using
the the random features method.

We conducted experiments on the MNIST dataset. We use both the Gaussian kernel (with
$\sigma=8.5$) 
and the polynomial kernel $k(\x,\z) = (0.01 \x^\T\z + 1)^3$. The regularization parameter was set
to $\lambda=0.01$. Running time were measured on a single c4.8xlarge EC2 instance.

Results are shown in Figure~\ref{fig:mnist_vs_rf}. Inspecting the error rates (left plots), we see
that while the random features method is able to deliver close to optimal error rates, there is
a clear gap between the errors obtained using our method and the ones obtained by the random
features method even for a very large number of random features. {\em The gap persists even
  if we set the number of random features to be as large as the training size!}

In terms of the running time (right plots), our method is, as expected, generally more expensive
than the sketch-and-solve approach of the random features method, at least when the number of random features is the same.
However, in order to get close to the performance of our method, the random features method
requires so many features that its running time eventually surpasses that of our method
with the optimal number of features (recall that for our method the number of random features
only affects the running time, not the generalization performance.)

It is worth noting that the optimal training time for our method on MNIST was actually quite small:
less than 2 minutes for the Gaussian kernel, and less than 4 minutes for the polynomial kernel,
and this without sacrificing in terms of generalization performance by using an approximation.

In Figure~\ref{fig:err_vs} we further explore the complex interaction between algorithmic complexity, quantity of data
and predicative performance. In this set of experiments we use subsamples of the COVTYPE dataset (to a maximum
sample of 80\% of the data; the rest is used for testing and validation) and subsamples from the extended MNIST-8M dataset
(to a maximum of 450K data points). We compare the performance of our method as the number of samples increases
to the performance of the random features method with three different profiles for setting $s$:
$s=20000$ (an $O(n)$ training algorithm), $s=2.26\sqrt{nd}$ (an $O(n^2 d)$ algorithm) and
$s = 0.2n$ (an $O(n^3)$ time, $O(n^2)$ memory algorithm). For both datasets we plot both error
as a function of the datasize and error as a function of training time. The graphs clearly
demonstrate the superiority of our method.

\subsection{Resources and Running Time on a Cloud Service}

Our implementation is designed to leverage distributed processing using clusters of machines.
The wide availability of cloud-based computing services provides easy access to such
platforms on an on-demand basis.

To give an idea of the running time and resources (and as a consequence the cost)
of training models using our algorithm on public cloud services, we applied our method
to various popular datasets on EC2. The results are summarized in Table~\ref{tab:ec2}.
We can observe that using our method it is possible to train high-quality models
on datasets as large as one million data points in a few hours. 
We remark that while we did tune $\sigma$ and $\lambda$ somewhat, we did not attempt to
tune them to the best possible
values.



\subsection{Additional Experimental Results}
\label{sec:additional}

\begin{sidewaystable*}
{\small
\protect\caption{\label{tab:ec2} Running time and resources on EC2. The datasets in the top part are classification datasets, and
on the bottom they are regression datasets.}
\begin{center}
\begin{tabular}{|c|c|c|c|c|c|c|c|c|}
  \hline
      {\bf Dataset} & {\bf n}  & {\bf d} & {\bf Parameters} & {\bf Resources} & {\bf s} & {\bf Iterations} & {\bf Time (sec)} & {\bf Error Rate} \tabularnewline
      \hline
      \sc{GISETTE} & 6,000 & 5,000 & $\sigma=3, \lambda=0.01$ & 1 c4.large & 500 & 6 & 8.1 & 3.50\% \tabularnewline
      \hline
      \sc{ADULT} & 32,561 & 123 & $\sigma = 8, \lambda=0.01$ & 1 c4.4xlarge & 5,000 & 13 & 19.8 & 14.99\% \tabularnewline
      \hline
      \sc{IJCNN1} & 49,990 & 22 & $\sigma = 0.3, \lambda=0.01$ & 1 c4.8xlarge & 12,500 & 120 & 85.1 & 1.39\% \tabularnewline
      \hline
      \sc{MNIST} & 60,000 & 780 & $\sigma=8.5, \lambda=0.01$ & 1 c4.8xlarge & 10,000 & 85 & 115 & 1.33\% \tabularnewline
      \hline
      \sc{MNIST-400K} & 400,000 & 780 & $\sigma=8.5, \lambda=0.01$ & 8 r3.8xlarge & 35,000 & 99 & 1580 & 0.89\%\tabularnewline
      \hline
      \sc{MNIST-1M} & 1,000,000 & 780 & $\sigma=8.5, \lambda=0.01$ & 42 r3.8xlarge & 35,000 & 189 & 3820 & 0.72\% \tabularnewline
      \hline
      \sc{EPSILON} & 400,000 & 2,000 & $\sigma=8, \lambda=0.01$ & 8 r3.8xlarge & 20,000 & 12 & 526 & 10.21\% \tabularnewline
      \hline
      \sc{COVTYPE} & 464,809 & 54 & $\sigma=0.1, \lambda=0.01$ & 8 r3.8xlarge & 55,000 & 555 & 6180 & 4.13\% \tabularnewline
      \hline \hline
      \sc{YEARMSD} &  463,715 & 90 & $\sigma=3, \lambda=0.001$ & 8 r3.8xlarge & 15,000 & 15 & 289 & $4.58\times 10^{-3}$ \tabularnewline
\hline

\hline
\end{tabular}
\end{center}}

\vspace{0.2in}

{\small
\protect\caption{\label{tab:precond lambda} Running time and resources on EC2 when using different $\lambda$ in the preconditioner.}
\begin{center}
\begin{tabular}{|c|c|c|c|c|c|c|}
  \hline
      {\bf Dataset} &  {\bf Parameters} & {\bf Resources} & {\bf Precond $\lambda$ ($\lambda_p$)} & {\bf s} & {\bf Iterations} & {\bf Time (sec)}  \tabularnewline
      \hline
      \sc{GISETTE} &  $\sigma=3, \lambda=0.01$ & 1 c4.large & 0.1 & 500 & 6 & 8.1 \tabularnewline
      \hline
      \sc{ADULT} &  $\sigma = 8, \lambda=0.01$ & 1 c4.4xlarge &  & 5,000 & 17 & 20.7  \tabularnewline
      \hline
      \sc{IJCNN1} &  $\sigma = 0.3, \lambda=0.01$ & 1 c4.8xlarge &  0.1 & 10,000 & 52 & 55.5  \tabularnewline
      \hline
      \sc{MNIST} &  $\sigma=8.5, \lambda=0.01$ & 1 c4.8xlarge & 0.1 & 10,000 & 37 & 76.3  \tabularnewline
      \hline
      \sc{MNIST-400K} &  $\sigma=8.5, \lambda=0.01$ & 8 r3.8xlarge & 0.1 & 40,000 & 42 & 1060 \tabularnewline
      \hline
      \sc{MNIST-1M} &  $\sigma=8.5, \lambda=0.01$ & 42 r3.8xlarge & 0.1 & 40,000 & 77 & 1210 \tabularnewline
      \hline
      \sc{EPSILON} & $\sigma=8, \lambda=0.01$ & 8 r3.8xlarge &  0.1 & 20,000 &  18 & 547  \tabularnewline
      \hline
      \sc{COVTYPE} &  $\sigma=0.1, \lambda=0.01$ & 8 r3.8xlarge & 0.2 & 40,000 & 206 & 2960 \tabularnewline
      \hline \hline
      \sc{YEARMSD} &  $\sigma=3, \lambda=0.001$ & 8 r3.8xlarge & 0.01 & 15,000 & 20 & 289  \tabularnewline
\hline

\hline
\end{tabular}
\end{center}}
\end{sidewaystable*}

\begin{sidewaystable*}
{\small
\protect\caption{\label{tab:compare-hilbert} Comparison to high-performance random features solver based on block ADMM~\cite{AvronSindhwani14}.}
\begin{center}
\begin{tabular}{|c|c|c|c|c|c|c|}
  \hline
      {\bf Dataset} &  {\bf Resources} & {\bf ADMM - s} & {\bf ADMM - time}& {\bf ADMM - error}& {\bf Proposed - time} & {\bf Proposed - error} \tabularnewline
      \hline
      \sc{MNIST} &  1 c4.8xlarge & 15,000 & 102 & 1.95\% & 115 & 1.33\% \tabularnewline
      \hline
      \sc{MNIST-400K} &  8 r3.8xlarge & 100,000 & 1017 & 1.10\%& 1580 & 0.89\%\tabularnewline
      \hline
      \sc{EPSILON}  &  8 r3.8xlarge & 100,000 & 1823 & 11.58\% & 526  & 10.21\% \tabularnewline
      \hline
      \sc{COVTYPE}&  8 r3.8xlarge & 115,000 & 6640 & 5.73\% & 6180 &  4.13\% \tabularnewline
      \hline \hline
      \sc{YEARMSD} &  8 r3.8xlarge & 115,000 & 958 & $5.01\times 10^{-3}$  &  289 & $4.58\times 10^{-3}$ \tabularnewline
\hline

\hline
\end{tabular}
\end{center}}
\vspace{0.2in}

{\small
\protect\caption{\label{tab:noprecond} Comparison to CG without preconditioning. See text for details on cells marked {\sc FAIL}.}
\begin{center}
\begin{tabular}{|c|c|c|c|c|c|}
  \hline
      {\bf Dataset} & {\bf No Precond - its}  & {\bf No Precond - time} & {\bf s} & {\bf Proposed - its} & {\bf Proposed - time}  \tabularnewline
      \hline
      \sc{GISETTE} & 1 & 3.66 & 500 & 6 & 8.1 \tabularnewline
      \hline
      \sc{ADULT} &  369 & 52.9 & 5,000 & 13 & 19.1 \tabularnewline
      \hline
      \sc{IJCNN1} & 764 & 230 & 10,000 & 120 & 85.1 \tabularnewline
      \hline
      \sc{MNIST} & 979 & 500 & 10,000 & 85 & 115 \tabularnewline
      \hline
      \sc{MNIST-200K} & {\sc FAIL} & 3890 & 25,000 & 90 & 1090 \tabularnewline
      \hline
      \sc{MNIST-300K} & {\sc FAIL} & 5320 & 25,000 & 117 & 1400 \tabularnewline
      \hline
      \sc{EPSILON} &  194 & 854 & 20,000 & 12 & 526 \tabularnewline
      \hline
      \sc{COVTYPE} & 913 & 5111 & 55,000 & 555 & 6180 \tabularnewline
      \hline
      \hline
      \sc{YEARMSD} & {\sc FAIL} & 1900 & 15,000 & 15 & 289 \tabularnewline
\hline

\hline
\end{tabular}
\end{center}}
\end{sidewaystable*}


So far we described and experimented with using $\matZ\matZ^\T + \lambda \matI_n$ as a preconditioner. One straightforward idea
is to try to use  $\matZ\matZ^\T + \lambda_p \matI_n$ as a preconditioner, where $\lambda_p$ is now a parameter that is not
necessarily equal to $\lambda$. Although our current theory does not cover this case, we evaluate this idea empirically
and report the results Table~\ref{tab:precond lambda}. Our experiments indicate that setting $\lambda_p$ to be larger
than $\lambda$ often produces higher quality preconditioner. In particular, $\lambda_p = 10\lambda$ seems like a reasonable
rule-of-thumb for setting $\lambda_p$. We leave a theoretical analysis of this scheme to future work.

In Table~\ref{tab:compare-hilbert} we compare our method to an high-performance distributed block ADMM-based solver using
the Random Features Method~\cite{AvronSindhwani14}. We use the same resource configuration as in Table~\ref{tab:ec2}
(we remark that the ADMM solver is rather memory efficient and can function with less resources).
The ADMM solver is more versatile in the choice of objective function, so we use hinge-loss (SVM).
We use the same bandwidth ($\sigma$) and reguarization parameter ($\lambda$) and in Table~\ref{tab:ec2}.
In general we set the number of random features ($s$) to be equal 25\% of the dataset size. We clearly see that our
method achieves better error rates, usually with better running times.

In Table~\ref{tab:noprecond} we examine whether the preconditioner indeed improves convergence
and running time. We set the maximum iterations to 1000, and declare failure if failed
to converge to $10^{-3}$ tolerance for classification and $10^{-5}$ for regression.
We remark the following on the items labeled {\sc FAIL}:
\begin{itemize}
\item For {\sc MNIST-200k}, without preconditioning CG failed to convergence but the error rate of the final model was just as good as our method.

\item For {\sc MNIST-300K} the error rate of the final model deteriorated to 1.33\% (compare to 0.92\%).

\item For {\sc YEARMSD} the final residual was $6.63\times 10^{-4}$ and the error deteriorated
  to $5.25\times 10^{-3}$ (compare to $4.58\times 10^{-3}$).
\end{itemize}
Almost always (with two exception, one of them a tiny dataset) our method was faster
than the non-preconditioned method. More importantly our method is much more robust: the non-preconditioned algorithm
failed in some cases.

\begin{table}[t!]
{\small
\protect\caption{\label{tab:tol} Comparison between setting tolerance to $10^{-2}$ and $10^{-3}$.}
\begin{center}
\begin{tabular}{|c|c|c|c|c|}
  \hline
      {\bf Dataset} & {\bf $10^{-2}$ - its}  & {\bf $10^{-2}$ - Error Rate} & {\bf $10^{-3}$ - its} & {\bf $10^{-3}$ - Error Rate}  \tabularnewline
      \hline
      \sc{GISETTE} & 3 & 3.50 & 6 & 3.50\% \tabularnewline
      \hline
      \sc{ADULT} & 10 & 14.99\% &  13 & 14.99\% \tabularnewline
      \hline
      \sc{IJCNN1} & 69 & 1.38\% &  120 & 1.39\% \tabularnewline
      \hline
      \sc{MNIST} & 54 & 1.37\% & 85 & 1.33\% \tabularnewline
      \hline
      \sc{MNIST-200K} & 54 & 1.00\% &  90 & 0.99\% \tabularnewline
      \hline
      \sc{MNIST-300K} & 73 & 0.92\% &  117 & 0.92\% \tabularnewline
      \hline
      \sc{EPSILON} & 7 & 10.21\% &  12 & 10.22\% \tabularnewline
      \hline
      \sc{COVTYPE} & 253 & 4.12\% &  555 & 4.13\% \tabularnewline
\hline

\hline
\end{tabular}
\end{center}}
\end{table}

In Table~\ref{tab:tol} we examine how classification generalization quality is affected by choosing
a more relaxed tolerance criteria. In general, setting tolerance to $10^{-3}$ is
just barely better than $10^{-2}$ in terms of test error. Only in one case ({\sc MNIST}) the difference
is bigger than 0.01\%. Running time for
$10^{-3}$ by is worse by a small factor and results are almost the same.
We  set the tolerance to $10^{-3}$ to be consistent with the notion of exploiting data to
the fullest, although in practice $10^{-2}$ seems to be sufficient.

\section{Conclusions}
\label{sec:conclusions}

Kernel ridge regression is a powerful non-parametric technique whose solution has a closed form that involves the solution of a linear system,
and thus is amenable to applying advanced numerical linear algebra techniques.
A naive method for solving this system is too expensive to be realistic beyond ``small data''.
In this paper we propose an algorithm that solves this linear system to high accuracy using a combination
of sketching and preconditioning. Under certain conditions, the running time of our algorithm is
somewhere between $O(n^2)$ and $O(n^3)$, depending on properties of the data, kernel, and feature map.
Empirically, it often behaves like $O(n^2)$. As we show experimentally, our algorithm is highly effective on
datasets with as many as one million training examples.

Obviously the main limitation of our algorithm is the $\Theta(n^2)$ memory requirement for storing the
kernel matrices. There are a few ways in which our algorithm can be leveraged to allow learning on much
larger datasets. One non-algorithmic software-based idea is to use an out-of-core algorithm, i.e. use SSD storage, or even magnetic drive, to hold
the kernel matrix. From an algorithmic perspective there are quite a few options. One idea is to use boosting
to design a model that is an ensemble of a several smaller models based on non-uniform sampling of the data.
Huang et al. recently showed that this can be highly effective in the context of kernel ridge regression~\cite{HuangEtAl14}.
Another idea is to use our solver as the block solver in the block coordinate descent algorithm suggested by Tu et al.~\cite{TuEtAl16}.
We leave the exploration of these techniques to future work.

More importantly, one should note that even in the era of Big Data, it is not always the case that for a particular
problem we have access to very big training set. In such cases it is even more important to fully leverage
the data, and produce the best possible model. Our algorithm and implementation provide an effective
way to do so.

\section*{Acknowledgments}

The authors acknowledge the support from the XDATA program of the Defense
Advanced Research Projects Agency (DARPA), administered through Air
Force Research Laboratory contract FA8750-12-C-0323.

\bibliography{FasterKRR}
\bibliographystyle{plain}

\appendix

\section{Appendix - Proof of Theorem~\ref{thm:genmain}}

\begin{prop}
  \label{prop:genid}
  Let $\matK + \lambda \matI_n = \matL \matL^\T$ be a Cholesky decomposition of $\matK + \lambda \matI_n$. Let
  $\matM = \matL^{-1}$. Let ${\cal Q} = (\q_1, \dots, \q_n) \subset {\cal H}_k$ defined by
  $
  \q_i = \sum^n_{j=1}\matM_{ij}k(\x_j, \cdot)\,.
  $
  We have
  \begin{equation*}
  \matK({\cal Q}, {\cal Q}) + \lambda \matL^{-1} \matL^{-\T} = \matI_n\,.
  \end{equation*}
\end{prop}

\begin{proof}
  Since $\matK_{ij} = k(\x_i, \x_j) = \dotprod{k(\x_i, \cdot)}{k(\x_j, \cdot)}{{\cal H}_k}$ and inner products
  are bilinear, we have $\matK({\cal Q}, {\cal Q}) = \matL^{-1} \matK \matL^{-\T}$. So,
  \begin{equation*}
  \matI_n = \matL^{-1} (\matK + \lambda \matI_n) \matL^{-\T} = \matK({\cal Q}, {\cal Q}) + \lambda \matL^{-1} \matL^{-\T}\,.
  \end{equation*}
\end{proof}

\begin{prop}
  \label{prop:genstat_dim}
  Under the conditions of the previous proposition, $$\Trace{\matK({\cal Q}, {\cal Q})} = s_\lambda(\matK)\,.$$
\end{prop}

\begin{proof}
  Let $\lambda_1, \dots, \lambda_n$ be the eigenvalues of $\matK$.
    \begin{eqnarray*}
    \Trace{\matK({\cal Q}, {\cal Q})} & = & \Trace{\matI_n  - \lambda \matL^{-1} \matL^{-\T}} \\
    & = & n - \lambda \Trace{\matL^{-1} \matL^{-\T}} \\
    & = & n - \lambda \Trace{(\matK + \matI_n)^{-1}} \\
    & = & n - \sum^{n}_{i=1}\frac{\lambda}{\lambda_i + \lambda} \\
    & = & \sum^{n}_{i=1}\frac{\lambda_i}{\lambda_i + \lambda} \\
    & = &  \Trace{(\matK + \lambda I_n)^{-1}  \matK}
  \end{eqnarray*}
\end{proof}

\begin{proof}[Proof of Theorem~\ref{thm:genmain}]
  We prove that with probability of at least $1-\delta$
  \begin{equation}
    \label{eq:gengoal}
    \frac{2}{3}(\matZ \matZ^\T + \lambda \matI_n) \preceq \matK + \lambda \matI_n \preceq 2(\matZ \matZ^\T + \lambda \matI_n)\,,
  \end{equation}
  or equivalently,
  \begin{equation}
      \label{eq:gengoal1}
  \frac{1}{2}(\matK + \lambda \matI_n) \preceq \matZ \matZ^\T + \lambda \matI_n \preceq \frac{3}{2}(\matK + \lambda \matI_n)\,.
  \end{equation}
  Thus, with probability of $1-\delta$ the relevant condition number is bounded by $3$. For PCG, if the condition number
  is bounded by $\kappa$, we are guaranteed to reduce the error (measured in the matrix norm of the linear equation) to an
  $\epsilon$ fraction of the initial guess after $\lceil \sqrt{\kappa}\ln(2/\epsilon)/2 \rceil$
  iterations~\cite{Shewchuk94}. This immediately leads to the bound in the theorem statement.

  Let $\matK + \lambda \matI_n = \matL \matL^\T$ be a Cholesky decomposition of $\matK + \lambda \matI_n$. Let
  $\matM = \matL^{-1}$. Let ${\cal Q} = (\q_1, \dots, \q_n) \subset {\cal H}_k$ defined by
  $
  \q_i = \sum^n_{j=1}\matM_{ij}k(\x_j, \cdot)\,.
  $
  Let $\matZ_{\cal Q}\in\R^{n \times s}$ be the matrix whose row $i$ is equal to $\varphi(k(\x_i, \cdot))$. Due
  to the linearity of $\varphi$, we have
  $\varphi(\q_i) = \sum^n_{j=1}\matM_{ij}\varphi(k(\x_j, \cdot)) = \sum^n_{j=1}\matM_{ij}\z_j$
  so $\matZ_{\cal Q} = \matL^{-1} \matZ$.

  It is well known that for $\matC$ that is square and invertible $\matA \preceq \matB$ if and only if
  $\matC^{-1} \matA \matC^{-T} \preceq \matC^{-1} \matB \matC^{-T}$. Applying this to the previous equation
  with $\matC = \matL$ yields
  \begin{equation}
    \label{eq:gengoal2}
  \frac{1}{2}\matI_n \preceq \matZ_{\cal Q} \matZ^\T_{\cal Q} + \lambda \matL^{-1} \matL^{-\T} \preceq \frac{3}{2}\matI_n\,.
  \end{equation}

  A sufficient condition for~\eqref{eq:gengoal2} to hold is that
  \begin{equation}
    \label{eq:goal3}
    \TNorm{\matZ_{\cal Q} \matZ^\T_{\cal Q} + \lambda \matL^{-1} \matL^{-\T} - \matI_n} \leq \frac{1}{2}\,.
  \end{equation}
  According to Proposition~\ref{prop:genid} we have
  \begin{equation*}
  \TNorm{\matZ_{\cal Q} \matZ^\T_{\cal Q} + \lambda \matL^{-1} \matL^{-\T} - \matI_n} =
  \TNorm{\matZ_{\cal Q} \matZ^\T_{\cal Q} - \matK({\cal Q}, {\cal Q})}\,.
  \end{equation*}
  Now, the approximate multiplication property along with the requirement that $s \geq  f(s_\lambda(\matK)^{-1}/2, 0, \delta)$ or  $s \geq  f(s_\lambda(\matK)^{-1}/2\sqrt{2}, 1/2\sqrt{2}, \delta)$ guarantee that with probability of at least $1-\delta$ we have
  \begin{equation*}
  \TNorm{\matZ_{\cal Q} \matZ^\T_{\cal Q} - \matK({\cal Q}, {\cal Q})} \leq \FNorm{\matZ_{\cal Q} \matZ^\T_{\cal Q} - \matK({\cal Q}, {\cal Q})} \leq \frac{1}{2}\,.
  \end{equation*}
  Now complete the proof using the equality
  $\Trace{\matK({\cal Q}, {\cal Q})} = s_\lambda(\matK)$ (Proposition~\ref{prop:genstat_dim}).

\end{proof}

\end{document}